\documentclass[11pt]{article}
\usepackage{bbm}
\usepackage[all]{xy}
\usepackage{amsfonts,amssymb,amsmath,amscd,amsthm}
\usepackage{mathrsfs}
\usepackage{latexsym}
\usepackage{color}
\usepackage{verbatim,enumitem}
\usepackage{hyperref}
\input diagxy
\usepackage{tikz}
\usepackage[total={150mm,225mm}]{geometry}

\theoremstyle{plain}
  \newtheorem{thm}{Theorem}[section]
  \newtheorem{lem}[thm]{Lemma}
  \newtheorem{prop}[thm]{Proposition}
  \newtheorem{cor}[thm]{Corollary}
\theoremstyle{definition}
  \newtheorem{defn}[thm]{Definition}
  
  \newtheorem{exmp}[thm]{Example}
  \newtheorem{rem}[thm]{Remark}

\begin{document}
\newcommand\thda{\mathrel{\rotatebox[origin=c]{-90}{$\twoheadrightarrow$}}}
\newcommand\thua{\mathrel{\rotatebox[origin=c]{90}{$\twoheadrightarrow$}}}

\newcommand{\oto}{{\lra\hspace*{-3.1ex}{\circ}\hspace*{1.9ex}}}
\newcommand{\lam}{\lambda}
\newcommand{\da}{\downarrow}
\newcommand{\Da}{\Downarrow\!}
\newcommand{\D}{\Delta}
\newcommand{\ua}{\uparrow}
\newcommand{\ra}{\rightarrow}
\newcommand{\la}{\leftarrow}
\newcommand{\Lra}{\Longrightarrow}
\newcommand{\Lla}{\Longleftarrow}
\newcommand{\rat}{\!\rightarrowtail\!}
\newcommand{\up}{\upsilon}
\newcommand{\Up}{\Upsilon}
\newcommand{\ep}{\epsilon}
\newcommand{\ga}{\gamma}
\newcommand{\Ga}{\Gamma}
\newcommand{\Lam}{\Lambda}
\newcommand{\CF}{{\cal F}}
\newcommand{\CG}{{\cal G}}
\newcommand{\CH}{{\cal H}}
\newcommand{\CI}{{\cal I}}
\newcommand{\CB}{{\cal B}}
\newcommand{\CT}{{\cal T}}
\newcommand{\CS}{{\cal S}}
\newcommand{\CV}{{\cal V}}
\newcommand{\CP}{{\cal P}}
\newcommand{\CQ}{{\cal Q}}
\newcommand{\mq}{\mathcal{Q}}
\newcommand{\cu}{{\underline{\cup}}}
\newcommand{\ca}{{\underline{\cap}}}
\newcommand{\nb}{{\rm int}}
\newcommand{\Si}{\Sigma}
\newcommand{\si}{\sigma}
\newcommand{\Om}{\Omega}
\newcommand{\bm}{\bibitem}
\newcommand{\bv}{\bigvee}
\newcommand{\bw}{\bigwedge}
\newcommand{\lra}{\longrightarrow}
\newcommand{\tl}{\triangleleft}
\newcommand{\tr}{\triangleright}
\newcommand{\dda}{\downdownarrows}
\newcommand{\dia}{\diamondsuit}
\newcommand{\y}{{\bf y}}
\newcommand{\colim}{{\rm colim}}
\newcommand{\fR}{R^{\!\forall}}
\newcommand{\eR}{R_{\!\exists}}
\newcommand{\dR}{R^{\!\da}}
\newcommand{\uR}{R_{\!\ua}}
\newcommand{\swa}{{\swarrow}}
\newcommand{\sea}{{\searrow}}
\newcommand{\bbA}{{\mathbb{A}}}
\newcommand{\bbB}{{\mathbb{B}}}
\newcommand{\bbC}{{\mathbb{C}}}
\numberwithin{equation}{section}
\renewcommand{\theequation}{\thesection.\arabic{equation}}

\title{{Scott approach distance on metric spaces\footnote{This work is supported by National Natural Science Foundation of China (11771130).}}
\\ {\small Dedicated to  Robert Lowen on his 70th birthday}
}

\author{Wei Li\footnote{Chengdu Shishi High School, Wenmiaoqianjie 93, Chengdu, China, email: mathli@foxmail.com}, Dexue Zhang\footnote{School of Mathematics, Sichuan University, Chengdu, China, email: dxzhang@scu.edu.cn}}

\date{}

\maketitle

\begin{abstract} The notion of Scott  distance between points and subsets in a  metric space,   a metric analogy of the Scott topology on an ordered set, is introduced, making a metric space into an approach space. Basic properties of Scott  distance are investigated, including its topological coreflection and its relation to injective $T_0$ approach spaces. It is proved that the topological coreflection of the Scott   distance is sandwiched between the   $d$-Scott topology and the generalized Scott topology; and that every injective $T_0$ approach space  is a cocomplete and continuous metric space equipped with its Scott  distance.


 \end{abstract}

\section{Introduction}

In 1989,   Lowen \cite{RL89} introduced approach spaces as a common extension of metric spaces and topological spaces.
As explored in the monograph \cite{Lowen15}, approach spaces are closely related to many disciplines in mathematics, e.g. topology, analysis, probability, domain theory and etc.  This paper  focuses on one aspect of approach spaces, that is, their relation to metric spaces from the   viewpoint of domain theory. On one hand,  following Lawvere \cite{Lawvere73},  metric spaces (not necessarily symmetric) can be thought of as ordered sets valued in the  closed category $([0,\infty]^{\rm op},+)$. This point of view has led to the theory of quantitative domains, initiated by Smyth \cite{Smyth87,Smyth94}, with  metric spaces as  core objects, see e.g.  \cite{AR89,BvBR1998,FK97,FS02,Goubault,HW2011,HW2012,KW2011,Ru}. On the other hand, as advocated in  \cite{CHT,GH,Hof2011,Hof2013,HST}, approach spaces can be thought of as topological spaces valued in $([0,\infty]^{\rm op},+)$. This means that the theory of approach spaces is a theory of ``quantitative topological spaces". Thus, the relationship between metric spaces and approach spaces is analogous to that between ordered sets and topological spaces.

The interplay between order theoretic and topological properties of ordered sets is one of the main themes in domain theory  \cite{domains}. The Scott topology plays a prominent role in this regard. In 2000, Windels \cite{Windels} attempted to extend the theory of Scott topology to the metric setting, and succeeded in postulating the notion of Scott approach distance (Scott distance, for short) for algebraic metric spaces. But, the postulation in \cite{Windels} depends on the fact that an algebraic metric space has enough compact elements,  it is not applicable to a general metric space.

In this paper, we present a postulation of Scott distance for a general metric space, via help of Scott weights that are  a metric counterpart of Scott closed sets in ordered sets.  For an algebraic metric space, the Scott distance given here coincides with the one in  \cite{Windels}.

Basic properties of the Scott distance on metric spaces  are investigated in this paper, including its topological coreflection and its relation to injective $T_0$ approach spaces. It is shown that sending a metric to  its Scott distance yields a full embedding of the category of metric spaces and Yoneda continuous maps in the category of approach spaces. The topological coreflection of the Scott distance is a natural topology for a metric space, and it is sandwiched between the well-known  $d$-Scott topology and generalized Scott topology. Finally, it is shown that every injective $T_0$ approach space  is a cocomplete and continuous metric space equipped with its Scott distance, but, the converse fails.

\section{Preliminaries: metric spaces and approach spaces}
Following Lawvere \cite{Lawvere73}, by a metric space we mean a pair $(X,d)$ consisting of a set $X$ and a map $d:X\times X\lra [0,\infty]$ such that  $d(x,x)=0$ and $d(x,y)+d(y,z)\geq d(x,z)$ for all $x,y, z\in X$. The map $d$ is called a metric on $X$, the value $d(x,y)$   the distance from $x$ to $y$. Such spaces are also known as generalized metric spaces, pseudo-quasi-metric spaces, and hemi-metric spaces.

A non-expansive map $f: (X,d_X)\lra (Y,d_Y)$ between metric spaces is a map $f:X\lra Y$ such that $d_X(x,y)\geq d_Y(f(x),f(y))$ for all $x, y$ in $X$. A map $f: (X,d_X)\lra (Y,d_Y)$ between metric spaces is isometric if $d_X(x,y)=d_Y(f(x),f(y))$ for all $x,y\in X$.

Metric spaces and non-expansive maps form a complete category,   denoted by ${\sf Met}$.   In particular, the product of  a family of metric spaces $ (X_i,d_i), i\in I$,   is given by the set $\prod_i X_i$ equipped with the  metric $d((x_i),(y_i))=\sup_{i\in I}d_i(x_i,y_i)$.

An approach space \cite{RL89,RL97} is a pair $(X,\delta)$ consisting of a set $X$ and  a map $\delta:X\times 2^X \lra [0,\infty]$,   called an approach distance (distance, for short) on $X$, subject to the following conditions:   for all $x\in X$  and $ A,B\in 2^X$,
\begin{enumerate}[label={\rm(A\arabic*)}] \setlength{\itemsep}{-2pt}
\item  $\delta(x, \{x\})=0$;
\item  $\delta(x, \varnothing)=\infty$;
\item  $\delta(x, A\cup B)=\min\{\delta(x, A),\delta(x, B)\}$;
\item   $\delta(x,A)\leq\delta(x,B)+\sup_{b\in B}\delta(b,A)$. \end{enumerate}

In the original definition of approach spaces \cite{RL89,RL97}, instead of (A4), the following condition is used: \begin{enumerate}\item[\rm (A4')]  For all $\varepsilon\in[0,\infty]$, $\delta(x, A)\leq \delta(x, A^\varepsilon)+ \varepsilon$, where $A^\varepsilon=\{x\in X|\ \delta(x,A) \leq\varepsilon\}$. \end{enumerate} It is easily seen that, in the presence of (A1)--(A3),   (A4') is equivalent to (A4).

A contraction $f: (X,\delta_X)\lra (Y,\delta_Y)$ between approach spaces is a map $f: X\lra Y$ such that $\delta_X(x,A)\geq \delta_Y(f(x),f(A))$ for all $A\subseteq X$ and $x\in X$.  Approach spaces and contractions form a topological category (see \cite{AHS} for topological categories), denoted by  ${\sf App}$.

A metric space is an ordered set (or, a category) valued in the  closed category $\mathfrak{L}=([0,\infty]^{\rm op},+)$; an approach space is a topological space  valued in $\mathfrak{L}$. So, the relationship between approach spaces and metric spaces is analogous to that between topological spaces and ordered sets,  as emphasized in \cite{CHT,GH,Hof2011,Hof2013,LZ16}.

Let $2$ denote the  closed category $(\{0,1\},\min)$.
The map $\omega: 2\lra \mathfrak{L}$ sending $1$   to $0$   and  $0$   to $\infty$   and the map $\iota: \mathfrak{L}\lra 2$  sending   $0$   to $1$  and  all $x\in(0,\infty]$ to $0$  are both lax monoidal functors. So, they induce a pair of functors between the category {\sf Ord} of  ordered sets (as categories valued in $2$) and order-preserving maps and the category of metric spaces and non-expansive maps: \[\omega:{\sf Ord}\lra{\sf Met}\]
 and \[\iota:{\sf Met}\lra{\sf Ord}.\]
The functor $\omega$ maps an ordered set $(X,\leq)$ to the metric space $(X,\omega(\leq))$,  where \[\omega(\leq)(x,y)=\begin{cases}0, & x\leq y, \\ \infty, & \text{otherwise}. \end{cases}\] The functor $\iota$ maps a metric space $(X,d)$ to the ordered set $(X,\leq_d)$, where, \[x\leq_d y\iff d(x,y)=0.\] The functor $\omega:{\sf Ord}\lra{\sf Met}$ is full and faithful. Since, as order-preserving maps, $\omega: 2\lra \mathfrak{L}$ is left adjoint to $\iota: \mathfrak{L}\lra 2$,   the induced functors $\omega:{\sf Ord}\lra{\sf Met}$ and $\iota:{\sf Met}\lra{\sf Ord}$ form an adjunction with $\omega$ being the left adjoint.

The lax monoidal functors $\omega: 2\lra \mathfrak{L}$ and $\iota: \mathfrak{L}\lra 2$ also induce  an adjunction between the categories of topological spaces and approach spaces. Given a topological space $(X,\CT)$, the map  $\omega(\CT):X\times 2^X\lra[0,\infty]$, given by \[\omega(\CT)(x,A)=\begin{cases}0, & \text{if $x$ is in the closure of $A$}, \\ \infty, & \text{otherwise},\end{cases}\]
is an approach distance on $X$. The correspondence $(X,\CT)\mapsto(X,\omega(\CT))$ defines a full and faithful functor \[\omega:{\sf Top}\lra{\sf App}.\] Given an approach space $(X,\delta)$, the operator  on the powerset of $X$ given by \[x\in \overline{A}\iff \delta(x,A)=0 \]   is the closure operator for a topology, denoted by $\iota(\delta)$, on $X$. This process gives a functor \[\iota:{\sf App}\lra{\sf Top}\] that is   right adjoint to $\omega:{\sf Top}\lra{\sf App}$.  The topology $\iota(\delta)$ is called the  topological coreflection  of $\delta$  \cite{RL97}.

Note that we use the same symbol for both of the functors ${\sf Ord}\lra{\sf Met}$ and ${\sf Top}\lra{\sf App}$, because it is easily detected from the context which one is meant. Likewise, we use the same symbol for both of the functors ${\sf Met}\lra{\sf Ord}$ and ${\sf App}\lra{\sf Top}$.

A metric space $(X,d)$ is said to be separated if $x=y$ whenever $d(x,y)=d(y,x)=0$. The opposite $d^{\rm op}$ of a metric $d$ on $X$ is defined to be the metric given by $d^{\rm op}(x,y)=d(y,x)$.

\begin{exmp}[The Lawvere metric, \cite{Lawvere73}]  For all $a,b$ in $[0,\infty]$, the Lawvere distance, $d_L(a,b)$, from $a$ to $b$ is defined to be the truncated minus $b\ominus a$, i.e., \[d_L(a,b)=b\ominus a=\max\{
b-a, 0\},\]
where, we take by convention that $\infty-\infty=0$ and $\infty-a=\infty$ for all $a<\infty$. It is clear that $([0,\infty],d_L)$ is a separated, non-symmetric, metric space. The opposite of the Lawvere metric is denoted by $d_R$, i.e., $d_R(x,y)=x\ominus y$.
\end{exmp}

The approach space $\mathbb{P}$ in the following example plays an important role in the theory of approach spaces. This space is closely related to the metric $d_R$ on $[0,\infty]$ (see Example \ref{P=Sigma}).

\begin{exmp}\label{real approach} (\cite{RL97,Lowen15})  For all $x\in[0,\infty]$ and $A\subseteq[0,\infty]$, let
$$\delta_\mathbb{P}(x,A)= \left\{\begin{array}{ll} x\ominus\sup A,  & A\neq\emptyset,\\
\infty, &A=\emptyset.
\end{array}\right.$$ Then $\delta_\mathbb{P}$ is an approach  distance on $[0,\infty]$.   The space $([0,\infty],\delta_\mathbb{P})$ is denoted by $\mathbb{P}$.
\end{exmp}

As for topological spaces, approach spaces can be described in many ways \cite{RL97,Lowen15}. One of them we need   is the description by \emph{regular functions}.  A   regular function  of an approach space $(X,\delta)$ is a contraction $\phi:(X,\delta)\lra\mathbb{P}$, where $\mathbb{P}$ is the  approach space given in Example \ref{real approach}. Explicitly, a regular function of $(X,\delta)$ is a function $\phi:X\lra[0,\infty]$ such that \[\delta(x,A)\geq \phi(x)\ominus \sup \phi(A)\] for all $x\in X$ and all $A\subseteq X$.

Condition (A4) in the definition of approach spaces ensures that for each $A\subseteq X$, $\delta(-,A)$ is a regular function  of $(X,\delta)$.

\begin{thm} {\rm(\cite{RL97})} \label{regular functions} Let $(X,\delta)$ be an approach space. Then the set $\mathcal{R}X$ of regular functions of $(X,\delta)$   satisfies the following conditions:
\begin{enumerate}[label={\rm(R\arabic*)}] \setlength{\itemsep}{-2pt}
 \item  For each subset $\{\phi_i\}_{i\in I}$ of $\mathcal{R}X$,  $\sup_{i\in I}\phi_i\in \mathcal{R}X$.
\item  For all   $\phi,\psi\in\mathcal{R}X$,  $\min\{\phi,\psi\}\in \mathcal{R}X$.
\item  For all $\phi\in\mathcal{R}X$ and $r\in[0,\infty]$, both $\phi+r$ and $\phi\ominus r$ are in $\mathcal{R}X$.
\end{enumerate}

Conversely, suppose that $\mathcal{S}\subseteq [0,\infty]^X$ satisfies the conditions {\rm(R1)--(R3)}. Define a map $\delta: X\times 2^X\lra [0, \infty]$ by \begin{equation}\delta(x,A)=\sup\{\phi(x)\mid \phi\in\mathcal{S}~{\rm and}~  \phi(a)=0 ~{\rm for~all}~ a\in A\}.\end{equation}  Then $(X,\delta)$ is an approach space with   $\mathcal{S}$ being its set of regular functions. \end{thm}
We leave it to the reader to check that for each approach space $(X,\delta)$, the closed sets of its topological coreflection $\iota(\delta)$ are  given by $\{\phi^{-1}(0)\mid \phi\in \mathcal{R}X\}.$

Contractions between approach spaces can  be characterized in terms of regular functions.
\begin{prop}\label{contraction by regular frame} {\rm(\cite{RL97})} If $(X,\delta)$ and $(Y,\rho)$ are approach spaces and $f:X\lra Y$, then $f$ is a contraction if and only if for each $\phi\in\mathcal{R}Y$, $\phi\circ f \in\mathcal{R}X$. \end{prop}

Let $(X,\delta)$ be an approach space. A subset $\mathcal{B}\subseteq[0,\infty]^X$ is a \emph{subbasis} for the regular functions of $(X,\delta)$ if $\mathcal{R}X$ is the smallest set that contains $\mathcal{B}$ and satisfies {\rm(R1)--(R3)}.

\begin{prop}{\rm(\cite{RL97})} \label{initial source} Let $(X,\delta)$ be an approach space.  The source \[\{\delta(-,A): (X,\delta)\lra \mathbb{P}\}_{A\in 2^X}\] is initial.
Hence $\{\delta(-,A)\mid A\subseteq X\}$ is a  subbasis for the regular functions of $(X,\delta)$.
\end{prop}

\begin{cor}\label{subbasis of product} For each family of approach spaces $\{(X_i,\delta_i)\}_{i\in I}$,  \[\{\delta_i(-,A_i)\circ p_i\mid i\in I, A_i\subseteq X_i \}\]  is a subbasis for the regular functions of the product space $ \prod_{i\in I}(X_i,\delta_i)$, where $p_i$ denotes the   projection on the $i$th coordinate.    \end{cor}

An order on a set generates many topologies, for instance, the Alexandroff topology, the Scott topology, the Lawson topology, and etc. So, one might expect   that a metric on a set $X$ will induce many approach distances. This is true. The first example is  the Alexandroff distance \cite{RL97,Windels}. This paper concerns the second one, the Scott distance. While the Alexandroff distance is a metric analogy of the Alexandroff topology,  the Scott distance is a metric analogy of the Scott topology.

Let $(X,d)$ be a metric space. A weight (a.k.a. a left module) \cite{KS05,Lawvere73,SV2005}  of $(X,d)$ is a function $\phi:X\lra[0,\infty]$ such that $\phi(x)\leq \phi(y)+d(x,y)$ for all $x,y\in X$. A coweight (a.k.a. a right module) of $(X,d)$ is a function $\psi:X\lra[0,\infty]$ such that $\psi(y)\leq \psi(x)+d(x,y)$ for all $x,y\in X$.
Said differently, a weight of $(X,d)$ is a non-expansive map $\phi:(X, d)\lra([0,\infty],d_R)$ and a coweight of $(X,d)$ is a non-expansive map $\psi:(X, d)\lra([0,\infty],d_L)$.

The set $\CP X$  of all weights of a metric space $(X,d)$   has the following properties: \begin{enumerate}[label={\rm(W\arabic*)}] \setlength{\itemsep}{-2pt} \item For each $x\in X$, $d(-,x)\in\CP X$. Such  weights are said to   be  representable. \item  For each family $\{\phi_i\}_{i\in I}$ of weights of $(X,d)$,  both $\inf_{i\in I}\phi_i$ and $\sup_{i\in I}\phi_i$ are in $\mathcal{P}X$.
\item  For all $\phi\in\mathcal{P}X$ and $r\in[0,\infty]$, both $\phi+r$ and $\phi\ominus r$ are in $\mathcal{P}X$.
\end{enumerate}
Therefore, $\CP X$ satisfies the conditions (R1)--(R3) in Theorem \ref{regular functions} and determines an approach distance on $X$ via \[\Gamma(d)(x,A)=\sup\{\phi(x)\mid \phi\in\CP X~{\rm and}~  \phi(a)=0 ~{\rm for~all}~ a\in A\}.\] It is easy to check that \[\Gamma(d)(x,A)=\left\{\begin{array}{ll}\infty, &A=\emptyset,\\ \inf\limits_{y\in A}d(x,y),& A\not=\emptyset. \end{array}\right.\]The distance $\Gamma(d)$ is called the  Alexandroff distance  on $(X,d)$ \cite{RL97,Windels}.
The correspondence $(X,d)\mapsto(X,\Gamma(d))$ defines a full and faithful functor \[\Gamma:{\sf Met}\lra {\sf App}.\]

The functor $\Gamma:{\sf Met}\lra {\sf App}$ has a right adjoint \[\Omega:{\sf App}\lra{\sf Met}  \]
which sends every approach space $(X,\delta)$ to the metric space $(X,\Omega(\delta))$ with  $\Omega(\delta)(x,y)= \delta(x,\{y\})$  \cite{RL97,Lowen15}. The metric $\Omega(\delta)$ is  called  the specialization metric  \cite{Windels} of $(X,\delta)$ because of its analogy to the specialization order of topological spaces, as shown in the commutative squares:
\[\bfig \square<700,500>[{\sf Ord}` {\sf Top}` {\sf Met}  `{\sf App}; \Gamma` \omega ` \omega ` \Gamma] \square(1500,0)/>`<-`<-`>/<700, 500>[{\sf Top} `{\sf Ord} `{\sf App}`{\sf Met} ;  \Omega` \iota`\iota`\Omega]
\efig\] where,  the functor  $\Gamma:{\sf Ord}\lra {\sf Top}$ sends each  ordered set $(X,\leq)$ to its Alexandroff topology and  $\Omega:{\sf Top}\lra {\sf Ord} $ sends each topological space to its specialization order.

For each metric space $(X,d)$, the topological coreflection of its Alexandroff distance is a natural topology for $(X,d)$, which is in fact the open ball topology on $(X,d)$. The open ball topology \cite{Goubault} on $(X,d)$ is the topology generated as a basis by the  open balls in $(X,d)$, where, for each $x\in X$ and $r>0$, a point $y\in X$ lies in the open ball $B(x,r)$ with center $x$ and radius $r$ if the distance $d(x,y)$ from $x$ to $y$ is less than $r$, i.e., \[B(x,r)=\{y\in X\mid d(x,y)<r\}.\]  A point $x$ belongs to the closure of a subset $A$ with respect to the open ball topology if and only if $\inf_{a\in A}d(x,a)=0$. So the topological coreflection of its Alexandroff distance is exactly the open ball topology on $(X,d)$ \cite{RL97,Windels}.

\begin{defn}
Let  $f:(X,d_X)\lra(Y,d_Y)$ and $g:(Y,d_Y)\lra(X,d_X)$ be non-expansive maps between  metric  spaces. We say that $f$ is   left adjoint to $g$ (or, $g$ is  right adjoint to $f$), if \[d_Y(f(x),y)= d_X(x,g(y))\] for all $x\in X$ and $y\in Y$. \end{defn}  Left and right adjoint non-expansive maps are a special case of left and right  adjoint functors between enriched categories \cite{Kelly,Lawvere73}, respectively.

For any $\phi,\psi\in\CP X$, let \begin{equation*}\label{metric on PX}\overline{d}(\phi,\psi)=\sup_{x\in X}d_L(\phi(x), \psi(x))= \inf\{r\mid\psi\leq\phi+ r\}. \end{equation*} Then $\overline{d}$ is a separated metric on $\CP X$.

\begin{lem}{\rm(Yoneda lemma, \cite{Lawvere73})} Let $(X,d)$ be a  metric space. Then $\overline{d}(d(-,x),\phi)=\phi(x)$ for all $x\in X$ and $\phi\in\CP X$. \end{lem}

Given a  metric space $(X,d)$, define \[\y:(X,d)\lra(\CP X,\overline{d})\] by $\y(x)=d(-,x)$  for all $x\in X$. Then $\y$ is an isometry by the Yoneda lemma, hence it is called the \emph{Yoneda embedding}.

Let $f:(X,d_X)\lra(Y,d_Y)$ be a non-expansive map between metric spaces. If $\phi$ is a weight of $(X,d_X)$ then the map $$\overline{f}(\phi):Y\lra[0,\infty], \quad  \overline{f}(\phi)(y) =\inf_{x\in X}(\phi(x)+d_Y(y,f(x)))$$ is a weight of $(Y,d_Y)$.  If $\psi$ is a weight   of $(Y,d_Y)$  then $\psi\circ f$ is a weight   of $(X,d_X)$. Similarly, if $\varphi$ is a  coweight  of $(Y,d_Y)$,  then $\varphi\circ f$ is a  coweight  of $(X,d_X)$;   if $\xi$ is a coweight of $(X,d_X)$ then the map $$\overline{f}(\xi):Y\lra[0,\infty], \quad  \overline{f}(\xi)(y) =\inf_{x\in X}(\xi(x)+d_Y(f(x),y))$$ is a coweight of $(Y,d_Y)$.

The following lemma is   a special case of Kan extensions  in (enriched) category theory, see e.g. \cite{Kelly,Lawvere73}.

\begin{lem} \label{left kan} Let $f:(X,d_X)\lra(Y,d_Y)$ be a non-expansive map between metric spaces. Then  $\overline{f}:(\CP X,\overline{d_X})\lra(\CP Y, \overline{d_Y})$ is left adjoint to $(-)\circ f:(\CP Y, \overline{d_Y})\lra(\CP X,\overline{d_X})$.
 \end{lem}

\section{The Scott distance on metric spaces}

A subset of an ordered set $(X,\leq)$ is Scott closed if it is a lower set and is closed under joins of directed subsets. Scott closed sets can also be characterized as lower subsets that are closed under least eventual upper bounds of eventual monotone nets. The Scott topology on $(X,\leq)$ is  the topology with Scott closed sets acting as the family of closed sets  \cite{domains}. Scott distance on metric spaces is an analogy of  Scott topology on ordered sets.  To our knowledge, in 2000, Windels \cite{Windels} made the first attempt to find such an analogy, resulting in the notion of Scott distance  for algebraic metric spaces. In this section, we present a postulation of this notion for a general metric space. In our approach, forward Cauchy nets take the role of eventual monotone nets,  Yoneda limits take the role of least eventual upper bounds, flat weights take the role of directed subsets, colimits of flat weights take the role of joins of directed subsets, and Scott weights that of Scott closed sets.  In the next section, we shall see that for algebraic metric spaces, the postulation given here coincides with the one of Windels.

A net $\{x_i\}_i$ in a metric space $(X,d)$ is forward Cauchy  \cite{BvBR1998,Wagner97} if $$\inf_i\sup_{k\geq j\geq i}d(x_j,x_k)=0.$$

Non-expansive maps  clearly preserve forward Cauchy nets. This fact will be used later.

\begin{defn}(\cite{BvBR1998,Wagner97})
Let $\{x_i\}_i$  be a forward Cauchy net in a metric space $(X,d)$. An element   $x\in X$  is a Yoneda limit\footnote{From the viewpoint of category theory,  \emph{Yoneda colimit} (\emph{Yoneda cocomplete}, resp.) will be a more appropriate terminology than \emph{Yoneda limit} (\emph{Yoneda complete}, resp.), because it is actually a colimit (cocomplete with respect to certain class of weights, resp.), see Proposition \ref{Yoneda limit=colimit} below. The reason for choosing   \emph{Yoneda limit} is to keep with the tradition in domain theory \cite{domains,Goubault}.  } (a.k.a. liminf) of $\{x_i\}_i$ if  for all $y\in X$, \[d(x,y)= \inf_i\sup_{j\geq i}d(x_j,y).\] \end{defn}

Yoneda limits are not necessarily unique. However, if both $x$ and $y$ are   Yoneda limit  of a  net $\{x_i\}_i$, then $d(x,y)=d(y,x)=0$. So, Yoneda limits in a separated metric space are unique.

\begin{defn} (\cite{BvBR1998,Wagner97}) A metric space  is Yoneda complete if each forward Cauchy net   has a Yoneda limit. \end{defn}

\begin{exmp}(\cite{Goubault}) \label{d_L} Both $([0,\infty],d_L)$ and $([0,\infty],d_R)$ are Yoneda complete.

If $\{x_i\}_i$ is a forward Cauchy net in $([0,\infty],d_L)$, then  $\{x_i\}$ is  eventually either a  constant net with value $\infty$ or  a Cauchy net of real numbers in the usual sense. In the first case, $\infty$ is a Yoneda limit of $\{x_i\}_i$; in the second case, the limit of the Cauchy net $\{x_i\}_i$ is a Yoneda limit of $\{x_i\}_i$.  Thus, $([0,\infty],d_L)$ is Yoneda complete.

If $\{x_i\}_i$ is a forward Cauchy net in $([0,\infty],d_R)$, then  $\{x_i\}_i$ converges in the usual sense (the limit can be $\infty$) and its limit is a Yoneda limit of $\{x_i\}_i$.  Thus, $([0,\infty],d_R)$ is Yoneda complete. \end{exmp}

For a forward Cauchy net $\{x_i\}_i$ in a metric space  $(X,d)$, it is clear that for each $y\in X$, $\{d(x_i,y)\}_i$ is a forward Cauchy net in $([0,\infty],d_R)$. So, $x$ is a Yoneda limit of $\{x_i\}_i$ if and only if for all $y\in X$, the net $\{d(x_i,y)\}_i$ converges to $d(x,y)$ (in the usual sense). This fact will be very useful.

The following important example of Yoneda complete metric spaces is contained in \cite[Proposition 7.14]{SV2005}, it is also  a special case of \cite[Theorem 3.1]{Wagner97}.
\begin{exmp}  \label{Yoneda limits in PX} For a metric space $(X,d)$, every forward Cauchy net $\{\phi_i\}_i$   in  $(\CP X,\overline{d})$ has a Yoneda limit, given by  $\phi(x)=\inf_i\sup_{j\geq i}\phi_j(x)$.\end{exmp}

A non-expansive map $f:(X,d_X)\lra(Y,d_Y)$ is  Yoneda continuous  if $f$ preserves Yoneda limits in the sense that if $x$ is a Yoneda limit of a forward Cauchy net $\{x_i\}_i$ then $f(x)$ is a Yoneda limit of $\{f(x_i)\}_i$. The category of  metric spaces and Yoneda continuous maps is denoted  by \[{\sf Met}^\ua.\] The full subcategory of ${\sf Met}^\ua$ consisting of Yoneda complete and separated metric spaces is a metric counterpart of the category of directed complete partially ordered sets in domain theory.

\begin{defn}  Let $(X,d)$ be  a metric space. A weight $\phi$ of $(X,d)$ is   a  Scott weight if for every forward Cauchy net $\{x_i\}_i$ of $(X,d)$ and every Yoneda limit $x$ of $\{x_i\}_i$,  \[\inf_i\sup_{j\geq i}\phi(x_j)\geq \phi(x).\]  \end{defn}

Scott weights are introduced in Wagner \cite{Wagner97} under the name \emph{Scott closed $\mathfrak{L}$-functors} from $(X,d^{\rm op})$ to $([0,\infty],d_L)$, where $\mathfrak{L}$ denotes Lawvere's quantale $([0,\infty]^{\rm op},+)$.

Let $\phi$ be a  weight and $\{x_i\}_i$ be a forward Cauchy net of a metric space $(X,d)$. If $x$ is a Yoneda limit of $\{x_i\}_i$, then \[\inf_i\sup_{j\geq i}d(x_j,x)=d(x,x)=0,\] hence \[\phi(x)=\inf_i\sup_{j\geq i}(d(x_j,x)+\phi(x))\geq \inf_i\sup_{j\geq i}\phi(x_j).\]
Therefore, the inequality in the definition of Scott weights is actually an equality. Furthermore, since $\{\phi(x_i)\}_i$ is a forward Cauchy net in $([0,\infty],d_R)$,  it  converges to a real number or infinity  in the usual sense, thus, \begin{equation*}\label{sup=inf}\inf_i\sup_{j\geq i}\phi(x_j) = \sup_i\inf_{j\geq i}\phi(x_j).\end{equation*} This proves the following
\begin{prop}\label{order convergence} For a  weight $\phi$ of a metric space $(X,d)$, the following are equivalent: \begin{enumerate}[label={\rm(\arabic*)}] \setlength{\itemsep}{-2pt} \item $\phi$ is   a  Scott weight. \item  For every forward Cauchy net $\{x_i\}_i$  and every Yoneda limit $x$ of $\{x_i\}_i$,  $\sup_i\inf_{j\geq i}\phi(x_j)\geq \phi(x)$. \item $\phi:(X,d)\lra([0,\infty],d_R)$ is Yoneda continuous. \end{enumerate}\end{prop}

\begin{prop}Let $(X,d)$ be  a metric space. Then
\begin{enumerate}[label={\rm(\arabic*)}] \setlength{\itemsep}{-2pt}
\item Every representable weight is  a Scott weight.
\item   For each family $\{\phi_i\}_i$ of Scott weights of $(X,d)$,   $\sup_i\phi_i$ is a Scott weight.
\item  For all Scott weights $\phi_1$ and $\phi_2$   of $(X,d)$,   $\min\{\phi_1,\phi_2\}$ is a Scott weight.
\item  For all Scott weight $\phi$ of $(X,d)$ and all $r \in[0,\infty]$,  both $\phi\ominus r$ and $\phi+r$ are Scott weights of $(X,d)$.
\end{enumerate}\end{prop}

Given a metric space $(X,d)$,  the collection of Scott weights of $(X,d)$ satisfies the conditions (R1)--(R3) in Theorem \ref{regular functions}, hence it determines an approach distance   $\si$  on $X$ via \begin{equation*}
\si(x,A)= \sup\{\phi(x)\mid \phi {\rm~is~a~Scott~weight~ and}~  \phi(a)=0 ~{\rm for~all}~ a\in A\}. \end{equation*} We call $\si$  the  Scott distance  of $(X,d)$ and write $\Sigma(X,d)$  for the approach space $(X,\si)$.

The following lemma shows that the metric information of $(X,d)$ is encoded in its Scott distance.
\begin{lem}\label{4.3} For each metric space $(X,d)$, $\Omega\Sigma(X,d)=(X,d)$.\end{lem}

\begin{proof} Write $\sigma$ for the distance of $\Sigma(X,d)$. We need to show $\si(x,\{y\})=d(x,y)$ for all $x,y\in X$. Since $d(-,y)$ is a Scott weight and $d(y,y)=0$, it follows that $\si(x,\{y\})\geq d(x,y)$ by definition of $\sigma$. Next, for every Scott weight $\phi$ with $\phi(y)=0$, since $\phi(x)\leq d(x,y)+\phi(y)=d(x,y)$, one obtains that \[\si(x,\{y\})= \sup\{\phi(x)\mid \phi {\rm~is~a~Scott~weight~  and}~ \phi(y)=0\}\leq d(x,y).\] Therefore, $\si(x,\{y\})=d(x,y)$. \end{proof}

The following conclusion   is a metric analogy of the fact that a map between ordered sets preserves directed joins if and only if it is continuous with respect to   Scott topology.

\begin{thm}\label{contraction=Yoneda continuous} A map $f:(X,d_X)\lra(Y,d_Y)$ between metric spaces is   Yoneda continuous  if and only if $f:\Sigma(X,d_X)\lra\Sigma(Y,d_Y)$ is a contraction. \end{thm}
\begin{proof}
\textbf{Necessity}. It suffices to show that for each Scott weight $\phi$ of $(Y,d_Y)$,   $\phi\circ f$ is a  Scott weight of $(X,d_X)$. This follows from the fact that a composite of Yoneda continuous maps is Yoneda continuous.

\textbf{Sufficiency}. Write $\sigma_X,\sigma_Y$ for the distances of  $\Sigma(X,d_X)$ and $\Sigma(Y,d_Y)$, respectively.   For all $x,y\in X$,  by Lemma \ref{4.3}, \[d_X(x,y)=\si_X(x,\{y\})\geq\si_Y(f(x),\{f(y)\})= d_Y(f(x),f(y)),\] hence $f$ is non-expansive.

It remains to show that $f$ preserves Yoneda limits.
Given a forward Cauchy net $\{x_i\}_i$ of $(X,d_X)$ and a Yoneda limit $x$ of $\{x_i\}_i$, we  show that $f(x)$ is a Yoneda limit of $\{f(x_i)\}_i$, that is, for all $y\in Y$, \[\inf_i\sup_{j\geq i}d_Y(f(x_j),y)=d_Y(f(x),y).\]

Since $f$ is a contraction, $d_Y(f(-),y)=d_Y(-,y)\circ f$ is a Scott weight of $(X,d)$, hence \[\inf_i\sup_{j\geq i}d_Y(f(x_j),y)\geq d_Y(f(x),y).\] Conversely,
\begin{align*}\inf_i\sup_{j\geq i}d_Y(f(x_j),y)&\leq \inf_i\sup_{j\geq i}(d_Y(f(x_j),f(x))+d_Y(f(x),y))\\ &\leq \inf_i\sup_{j\geq i}(d_X(x_j,x)+d_Y(f(x),y))\\ & =d_Y(f(x),y).\end{align*}

 This completes the proof.
\end{proof}

The correspondence $(X,d)\mapsto\Sigma(X,d)$ defines a full and faithful functor \[\Sigma: {\sf Met}^\ua\lra{\sf App}\] from the category of metric spaces and Yoneda continuous maps to the category of approach spaces. Moreover, the following square commutes:
\[\bfig \square<700,500>[{\sf Ord}^\ua` {\sf Top}` {\sf Met}^\ua  `{\sf App}; \Sigma` \omega ` \omega ` \Sigma]\efig\] where, ${\sf Ord}^\ua$ denotes the category of ordered sets and Scott continuous maps;  the functor $\Sigma: {\sf Ord}^\ua\lra{\sf Top}$  sends each ordered set to its Scott topology. Thus, Scott distance  on  metric spaces is an extension of Scott topology on ordered sets.

In order to present a  useful characterization of Scott weights, Proposition \ref{Scott via flat},  we need some other notions.

Let $(X,d)$ be a metric space. For each weight $\phi$ and each coweight $\psi$ of $(X,d)$, the tensor product of $\phi$ and $\psi$ \cite{SV2005} (a special case of composition of bimodules in \cite{Lawvere73}) is an element in $[0,\infty]$, given by \[\phi\otimes \psi=\inf_{x\in X}(\phi(x)+\psi(x)).\]

\begin{defn}(\cite{SV2005}) Let $(X,d)$ be a metric space,  a weight $\phi$ of $(X,d)$ is flat if $\inf_{x\in X}\phi(x)=0$ and
$\phi\otimes\max\{\psi_1,\psi_2\}=\max\{\phi\otimes \psi_1,\phi\otimes \psi_2\}$ for any coweights $\psi_1,\psi_2$ on $(X,d)$. \end{defn}

Every representable weight $d(-,x)$ is clearly flat.

Let $f:(X,d_X)\lra(Y,d_Y)$ be a non-expansive map between metric spaces. It is easy to check that for each weight $\phi$ of $(X,d_X)$ and each coweight $\psi$ of $(Y,d_Y)$,  \begin{equation*}
\overline{f}(\phi)\otimes\psi =\phi\otimes(\psi\circ f). \end{equation*}
The following conclusion is easily verified with help of this equation.

\begin{lem}\label{image of flat weight}  Let $f:(X,d_X)\lra(Y,d_Y)$ be a non-expansive map between metric spaces.   If $\phi\in\CP X$ is flat then so is $\overline{f}(\phi)$.
\end{lem}

For a weight $\phi$ of a metric space $(X,d)$, let \[\mathrm{B}^+\phi=\{(x,r)\in X\times [0,\infty)\mid \phi(x)<r\}.\] Define a binary relation $\sqsubseteq$ on $\mathrm{B}^+\phi$ by \[(x,r)\sqsubseteq(y,s)\iff r\geq s+d(x,y).\] It is clear that $\sqsubseteq$ is a reflexive and transitive relation. Indeed,  $(\mathrm{B}^+\phi,\sqsubseteq)$ is a subset of the well-known  ordered set $\mathrm{B}X$ of formal balls in $(X,d)$ \cite{Goubault}.

The equivalence of (1) and (3) in the following proposition is contained in \cite[Proposition 7.9 and Theorem 7.15]{SV2005}. A proof is included here for sake of completeness.
\begin{prop}  \label{flat weight} Let $(X,d)$ be a metric space and $\phi$ a weight of $(X,d)$.  The following   are equivalent:
\begin{enumerate}[label={\rm(\arabic*)}] \setlength{\itemsep}{-2pt} \item  $\phi$ is a flat weight.
\item   $\inf_{x\in X}\phi(x)=0$ and $(\mathrm{B}^+\phi,\sqsubseteq)$ is a directed set.
\item  There is a forward Cauchy net $\{x_i\}_i$ in $(X,d)$ such that $\phi=\inf_i\sup_{j\geq i} d(-,x_j)$.\end{enumerate}
\end{prop}

\begin{proof}

 $(1)\Rightarrow(2)$ We only need to check that $(\mathrm{B}^+\phi,\sqsubseteq)$ is  directed. Given $(x,r)$ and $(y,s)$ in $\mathrm{B}^+\phi$, consider the coweights $\psi_1=s+d(x,-)$ and $\psi_2=r+d(y,-)$. Since $\phi\otimes\psi_1=s+\phi(x) <s+r$ and $\phi\otimes\psi_2=r+\phi(y) <s+r$, it follows that $\phi\otimes\max\{\psi_1,\psi_2\}<r+s$. Then there is some $z\in X$ such that $\max\{\psi_1(z),\psi_2(z)\}+\phi(z)<r+s$, hence $\phi(z)+d(x,z)<r$ and $\phi(z)+d(y,z)<s$. Let $t=\min\{r-d(x,z),s-d(y,z)\}$. Then $\phi(z)<t$, $r\geq t+ d(x,z)$ and $s\geq t+d(y,z)$. This means that $(z,t)$ is an element in $\mathrm{B}^+\phi$ and is an upper bound of $(x,r)$ and $(y,s)$.

$(2)\Rightarrow(3)$ Write an element  in ${\rm B}^+\phi$ as a pair $(x_i,r_i)$ and define a net \[\mathfrak{x}:({\rm B}^+\phi,\sqsubseteq)\lra X\] by $\mathfrak{x}(x_i,r_i)=x_i.$
It is routine to check that $\mathfrak{x}$ is a forward Cauchy net and	that  for all $x\in X$,  \[ \phi(x)=\inf_{(x_i,r_i)}\sup_{(x_j,r_j) \sqsupseteq (x_i,r_i)}d(x,x_{j}). \]

$(3)\Rightarrow(1)$ First, we show that for every coweight $\psi$ of $(X,d)$, $\phi\otimes\psi=\inf_i\sup_{j\geq i}\psi(x_j)$.

On one hand, \begin{align*}\phi\otimes\psi& = \inf_x\Big( \inf_i\sup_{j\geq i}d(x,x_j)+\psi(x)\Big)\\ &= \inf_i\inf_x\sup_{j\geq i}(d(x,x_j)+\psi(x))\\ &\geq \inf_i\sup_{j\geq i}\inf_x(d(x,x_j)+\psi(x))\\ &= \inf_i\sup_{j\geq i}\psi(x_j). \end{align*}

On the other hand, for each $\varepsilon>0$, since $\{x_i\}_i$ is forward Cauchy, there is some $i_\varepsilon$ such that $d(x_i,x_j)\leq\varepsilon$ whenever $i_\varepsilon\leq i\leq j$. Then \begin{align*}\phi\otimes\psi& =   \inf_x\inf_i\sup_{j\geq i}(d(x,x_j)+\psi(x))\\ &\leq \inf_x\inf_{i\geq i_\varepsilon}\sup_{j\geq i}(d(x,x_j)+\psi(x))\\ &\leq \inf_x\inf_{i\geq i_\varepsilon}(d(x,x_i)+\varepsilon+\psi(x))\\ &= \inf_{i\geq i_\varepsilon}\psi(x_i)+\varepsilon \\ &\leq \inf_i\sup_{j\geq i}\psi(x_j)+\varepsilon. \end{align*} Therefore, $\phi\otimes\psi\leq\inf_i\sup_{j\geq i}\psi(x_j)$ by arbitrariness of $\varepsilon$.

Now for any coweights $\psi_1, \psi_2$ of $(X,d)$, \begin{align*}\max\{\phi\otimes \psi_1, \phi\otimes\psi_2\} &=\max\Big\{\inf_i\sup_{j\geq i}\psi_1(x_j),\inf_i\sup_{j\geq i} \psi_2(x_j)\Big\}  \\ &   = \inf_i\sup_{j\geq i}\max\{\psi_1(x_j),\psi_2(x_j)\}  \\ &  = \phi\otimes\max\{\psi_1,\psi_2\},  \end{align*} showing that $\phi$ is flat.
 \end{proof}

Let $(X,d)$ be a metric space and $\phi$ a weight of $(X,d)$. An element $a\in X$ is called a \emph{colimit} of $\phi$ if \begin{equation*}
\overline{d}(\phi,\y(y)) =d(a,y)\end{equation*} for all $y\in X$ \cite{BvBR1998,Ru}. In the language of enriched category theory, a colimit of $\phi$ is said to be a colimit of the identity map $(X,d)\lra(X,d)$ weighted by $\phi$ \cite{Kelly}.

\begin{prop} {\rm (\cite[Lemma 46]{FSW})} \label{Yoneda limit=colimit} For each forward Cauchy net $\{x_i\}_i$  in a metric space  $(X,d)$, an element $x$   is a Yoneda limit of  $\{x_i\}_i$ if and only if $x$ is a colimit of the   weight  $\phi=\inf_i\sup_{j\geq i} d(-,x_j)$. \end{prop}

Suppose that $f:(X,d_X)\lra(Y,d_Y)$ is a non-expansive map, $\{x_i\}_i$ is a forward Cauchy net  in  $(X,d)$, and that  $\phi=\inf_i\sup_{j\geq i} d(-,x_j)$ is the weight generated by $\{x_i\}_i$. It is not hard to check that the  weight of $(Y,d_Y)$ generated by the forward Cauchy net $\{f(x_i)\}_i$ is  $\overline{f}(\phi)$. This fact is indeed a special case of  \cite[Lemma 49]{FSW}. Thus,  a non-expansive map $f:(X,d_X)\lra(Y,d_Y)$ is Yoneda continuous if  and only if $f$ preserves colimits of flat  weights in the sense that if $x$ is a colimit of a flat weight $\phi$, then $f(x)$ is a colimit of the flat weight $\overline{f}(\phi)$. In particular, if $f:(X,d_X)\lra(Y,d_Y)$ is left adjoint to $g:(Y,d_Y)\lra(X,d_X)$, then $f$ is Yoneda continuous, because left adjoints preserve  all colimits.

\begin{prop}\label{Scott via flat} A weight $\phi$ of a metric space $(X,d)$ is   a  Scott weight if and only if for every flat weight $\psi$ of $(X,d)$ and every colimit $x$ of $\psi$, $\overline{d}(\psi,\phi)\geq\phi(x)$.  \end{prop}

\begin{proof}\textbf{Necessity}. Let $\psi$ be a flat weight of $(X,d)$ with a colimit $x$. By Proposition \ref{flat weight} $\psi=\inf_i\sup_{j\geq i}d(-,x_j)$ for some forward Cauchy net $\{x_i\}_i$. By Proposition \ref{Yoneda limit=colimit}, $x$ is a Yoneda limit of $\{x_i\}_i$. By Example \ref{Yoneda limits in PX}, $\psi$ is a Yoneda limit of $\{d(-,x_j)\}$ in $(\CP X,\overline{d})$, hence \[\overline{d}(\psi,\phi)=\inf_i\sup_{j\geq i} \overline{d}(d(-,x_j),\phi)=\inf_i\sup_{j\geq i}\phi(x_j)\geq \phi(x). \]

\textbf{Sufficiency}. Suppose $\{x_i\}_i$ is a forward Cauchy net with $x$ being a Yoneda limit. Then $\psi=\inf_i\sup_{j\geq i}d(-,x_j)$ is a flat weight with $x$ being a colimit.  Since $\psi$ is a Yoneda limit of $\{d(-,x_j)\}$ in $(\CP X,\overline{d})$, then  \[\inf_i\sup_{j\geq i}\phi(x_j)=\inf_i\sup_{j\geq i} \overline{d}(d(-,x_j),\phi)=\overline{d}(\psi,\phi)\geq\phi(x). \]
This completes the proof. \end{proof}

The Scott distance on a metric space is in general different from its Alexandroff distance. However, they coincide for the class of Smyth completable spaces. A metric space is  Smyth completable if every forward Cauchy net is biCauchy \cite{KS2002}. It is shown in  \cite[Proposition 6.5]{LZ16} that a metric space is  Smyth completable if and only if all of its flat weights are Cauchy, where a weight $\phi$ of  $(X,d)$ is Cauchy  \cite{Lawvere73} if there is a coweight $\psi$ of  $(X,d)$  such that $\phi\otimes \psi=0$ and $\phi(x)+\psi(y)\geq d(x,y)$ for all $x,y\in X$.

\begin{cor}\label{Smyth} For each Smyth completable metric space, the Scott distance coincides with the Alexandroff distance. \end{cor}
\begin{proof}This follows  from Proposition \ref{Scott via flat} and the observation that if a Cauchy weight of a metric space has a colimit then it is representable. \end{proof}

Smyth completability is not a necessary condition  for Scott distance to coincide  with Alexandroff distance. For example,   the space $([0,\infty),d_R)$  is not Smyth completable, but, its Scott distance is equal to its Alexandroff distance.

Let  $(X,d)$ be a metric space. The topological coreflection of the Alexandroff distance is a natural topology for  $(X,d)$ -- the open ball topology.  The topological coreflection of the Scott distance is  also a natural topology for  $(X,d)$, so, it  deserves a name.
\begin{defn}For a metric space $(X,d)$,  the topological coreflection of its Scott distance is called   the $c$-Scott topology on $(X,d)$.\end{defn}

It is clear that the $c$-Scott topology is coarser than the open ball topology, and they are equal for a Smyth completable metric space by Corollary \ref{Smyth}. In the following we discuss the relationship among the $c$-Scott topology, the $d$-Scott topology \cite{Goubault} and the generalized Scott topology \cite{BvBR1998} on metric spaces.
The main result asserts that the $c$-Scott topology is sandwiched between  the $d$-Scott topology and the generalized Scott topology.

\begin{defn} (\cite{BvBR1998}) A subset $U$ of a metric space $(X,d)$ is  generalized Scott open  if for every forward Cauchy net $\{x_i\}_i$ and every Yoneda limit $x$ of $\{x_i\}_i$, if $x\in U$ then there is some $\varepsilon>0$ and some index $i$ such that the open ball $B(x_j,\varepsilon)$ is contained in $U$ for all $j\geq i$.
The generalized Scott open subsets of $(X,d)$ form a topology, called the  generalized Scott topology  on $(X,d)$.\end{defn}

For a metric space $(X,d)$, let   \[\mathrm{B}X=\{(x,r)\mid x\in X, r\in [0,\infty)\}.\]   Define a binary relation $\sqsubseteq$ on $\mathrm{B}X$ by \[(x,r)\sqsubseteq(y,s)\iff r\geq s+d(x,y).\] The  ordered set $(\mathrm{B}X,\sqsubseteq)$  is called the set of formal balls in $(X,d)$, it plays an important role in the study of metric spaces, see e.g. \cite{Goubault,GN}. The ordered set $(\mathrm{B}^+\phi,\sqsubseteq)$ in Proposition \ref{flat weight} is  a subset of $(\mathrm{B}X,\sqsubseteq)$.

\begin{defn}(\cite{Goubault}) The $d$-Scott topology on a metric space $(X,d)$ is the topology on $X$ inherited from the Scott topology on the ordered set  $(\mathrm{B}X,\sqsubseteq)$ via the embedding   $\eta_X:X\lra \mathrm{B}X$ that sends each $x$ to   $(x,0)$. \end{defn}

\begin{thm}\label{sandwich}For each metric space $(X,d)$,  the $c$-Scott topology   is coarser than the generalized Scott topology and is finer than  the $d$-Scott topology. \end{thm}

In order to prove this conclusion, we make some preparations first. Suppose $D=\{(x_i,r_i)\}_i$ is a directed  set of $(\mathrm{B}X,\sqsubseteq)$. Define an order on the index set by $i\leq j$ if $(x_i,r_i)\sqsubseteq(x_j,r_j)$. Then the index set becomes a directed set, the resulting net $\{x_i\}_i$ is called the \emph{underlying net} of $D$. Since $r_i\geq r_j+d(x_i,x_j)$ whenever $(x_i,r_i)\sqsubseteq(x_j,r_j)$, it follows that $\{r_i\}_i$ converges to $r=\inf_ir_i$. So, for each $\varepsilon>0$, there is some index $i$ such that $d(x_j,x_k)\leq\varepsilon$ whenever $i\leq j\leq k$. In particular, the underlying net of $D$ is  forward Cauchy.

\begin{lem}{\rm(\cite[Lemma 7.4.25]{Goubault})} \label{7.4.25}
Let $\{(x_i,r_i)\}_i$ be a directed  set of $(\mathrm{B}X,\sqsubseteq)$. If $x$ is a Yoneda limit of the underlying net $\{x_i\}_i$  and $r=\inf_ir_i$, then $(x,r)$ is a join of  $\{(x_i,r_i)\}_i$ in $(\mathrm{B}X,\sqsubseteq)$.\end{lem}

In particular, for every element $x$ in a metric space $(X,d)$ and every $r\geq 0$, $\{(x,r+1/n)\}_{n\geq1}$ is a directed set in $(\mathrm{B}X,\sqsubseteq)$ with a join $(x,r)$.

\begin{proof}[Proof of Theorem \ref{sandwich}] First, we prove that the $c$-Scott topology is coarser than the generalized Scott topology. It suffices to check that for each Scott weight $\phi$ of $(X,d)$, the set $\{y\mid \phi(y)>0\}$ is generalized Scott open.

Let $\{x_i\}_i$ be a forward Cauchy net and $x$ be a Yoneda limit of $\{x_i\}_i$. Assume that $\phi(x)=r>0$.   By Proposition \ref{order convergence}, $\sup_i\inf_{j\geq i}\phi(x_j)\geq \phi(x)$. So, there is some index $i$ such that $\phi(x_j)\geq 3r/4$ whenever $j\geq i$. Let $\varepsilon=r/2$. For each $j\geq i$ and  $y\in B(x_j,\varepsilon)$, $\phi(y)\geq\phi(x_j)-d(x_j,y)\geq r/4>0$. So, the open ball $B(x_j,\varepsilon)$ is contained in  $\{y\mid \phi(y)>0\}$, showing that $\{y\mid \phi(y)>0\}$ is generalized Scott open.

Next, we prove that the $c$-Scott topology is finer than  the $d$-Scott topology. Given a Scott closed set $F$ in $(\mathrm{B}X,\sqsubseteq)$, define a map $\phi_F:X\lra[0,\infty]$ as follows: let  $\phi_F(x)=\inf\{r\mid (x,r)\in F\}$ if there is some $r\in[0,\infty)$ with $(x,r)\in F$; otherwise  let $\phi_F(x)=\infty$. If $\phi_F(x)<\infty$, then $(x,\phi_F(x))\in F$ since  $(x,\phi_F(x))$ is a join of the directed set $\{(x,\phi_F(x)+1/n)\}_{n\geq 1}$. Thus, we have $\phi_F^{-1}(0)=\eta_X(X)\cap F$. So, in order to prove the conclusion, we only need to show that  for each Scott closed set $F$ in $(\mathrm{B}X,\sqsubseteq)$, $\phi_F$ is a Scott weight of $(X,d)$. We do this in two steps.

\textbf{Step 1}.   $\phi_F$ is a weight. That is, $\phi_F(x)\leq \phi_F(y)+d(x,y)$ for all $x,y$. If $(y,s)\in F$, since $F$ is a lower set and $(x,s+d(x,y))\sqsubseteq(y,s)$, then $(x,s+d(x,y))\in F$, so, $\phi_F(x)\leq s+d(x,y)$. It follows that $\phi_F(x)\leq \phi_F(y)+d(x,y)$.

\textbf{Step 2}.   $\phi_F$ is a Scott weight. Let $\{x_i\}_{i\in I}$ be a forward Cauchy net and $a$ be a Yoneda limit of $\{x_i\}_{i\in I}$. Define a subset $D$ of $(\mathrm{B}X,\sqsubseteq)$ as follows: \[(x,r)\in D \iff \exists i\in I,~ x=x_i ~\text{and}~ \sup_{j\geq i}d(x_i,x_j)\leq r/2.\] For $(x_i,r), (x_j,s)\in D$, let $t=\min\{r/2,s/2\}$. Take some index $k$ such that $i,j\leq k$ and that $d(x_k,x_l)\leq t/2$ whenever $k\leq l$. Then $(x_k,t)\in D$ and it is an upper bound of $(x_i,r)$ and $(x_j,s)$ in $(\mathrm{B}X,\sqsubseteq)$, so, $D$ is a directed set.

We claim that $a$ is a Yoneda limit of the underlying net of $D$.
It suffices to check that for each $y\in X$, the net $\{d(x_i,y)\}_{(x_i,r)\in D}$ converges to $d(a,y)$. We check this in the case that $d(a,y)$ is finite.
Given  $\varepsilon>0$, since $a$ is a Yoneda limit of the forward Cauchy net $\{x_i\}_{i\in I}$,  the net $\{d(x_i,y)\}_{i\in I}$ converges to $d(a,y)$, so  there is some $i\in I$ such that $|d(x_j,y)-d(a,y)|<\varepsilon/2$ and that $d(x_i,x_j)\leq\varepsilon/2$ whenever $i\leq j$. By definition one has  $(x_i,\varepsilon)\in D$. We assert that for all $(x_j,s)\in D$, $|d(x_j,y)-d(a,y)|<\varepsilon$  whenever $(x_i,\varepsilon)\sqsubseteq(x_j,s)$. If $i\leq j$,  this is clear. If $i\not\leq j$, take an upper bound  $(x_k,t)$ of $(x_j,s)$ and $(x_i,\varepsilon)$ in $D$ with $i,j\leq k$, then $|d(x_k,y)-d(a,y)|<\varepsilon/2$ and $d(x_j,x_k)\leq s/2\leq\varepsilon/2$, hence $|d(x_j,y)-d(a,y)|\leq(d(x_j,y)-d(x_k,y)|+|d(x_k,y)-d(a,y)|<\varepsilon$. Therefore, the net $\{d(x_i,y)\}_{(x_i,r)\in D}$ converges to $d(a,y)$.

Let $\beta =\inf_i\sup_{j\geq i}\phi_F(x_j)$. For each $\varepsilon>0$, let \[D_\varepsilon=\{(x_i,r+\beta+\varepsilon)\mid (x_i,r)\in D\}.\] Then $D_\varepsilon$ is a directed set and it is eventually in the set $F$. By Lemma \ref{7.4.25}, $(a, \beta+\varepsilon)$ is a join of $D_\varepsilon$, hence $(a, \beta+\varepsilon)\in F$. Again by Lemma \ref{7.4.25}, we obtain that $(a,\beta)\in F$. Therefore, $\phi_F(a)\leq \beta= \inf_i\sup_{j\geq i}\phi_F(x_j)$, showing that $\phi_F$ is a Scott weight. \end{proof}

Theorem \ref{sandwich} implies, in particular, that for each metric space, the $d$-Scott topology is coarser than the generalized Scott topology. It should be noted that in the case of Yoneda complete metric spaces, a proof of this fact is contained in \cite[Exercise 7.4.51]{Goubault}.

\begin{exmp}\label{c not= d} The $c$-Scott topology is in general different from the $d$-Scott topology. Consider the metric space $(X,d)$ given in \cite[Remark 2.3]{GN}. That is, $X=[0,1]$ and \[d(x,y)=\begin{cases} |x-y|, & x, y\not=0,\\ 0, & y=0,\\ 1, & x=0, y>0.\end{cases}\]

It is not hard to see that $\{(1/2^n, 1/2^n)\}_{n\geq 1}$ is a directed set in $(\mathrm{B}X,\sqsubseteq)$ with a join $(0,0)$. So, the interval $(0,1]$ is not closed in the $d$-Scott topology, because for each Scott closed set $F$ in $(\mathrm{B}X,\sqsubseteq)$, if $(0,1]\subseteq F\cap \eta_X(X)$, then $F$  contains  $\{(1/2^n, 1/2^n)\}_{n\geq 1}$, hence $(0,0)$. But,  $(0,1]$ is closed in the $c$-Scott topology. To see this, notice  that if a forward Cauchy net $\{x_i\}_i$ in $(X,d)$ has a Yoneda limit then $\{x_i\}_i$ is  either an eventually  constant net with value $0$ or a convergent net (in the usual sense) with a limit not $0$.
Define $\phi:X\lra [0,\infty]$ by $\phi(0)=1$ and $\phi(x)=0$ whenever $x>0$. Then $\phi$ is a Scott weight of $(X,d)$ and $\phi^{-1}(0)=(0,1]$, hence $(0,1]$ is closed in the $c$-Scott topology, as desired. \end{exmp} 

The final result in this section gives a sufficient condition for the $c$-Scott topology of a metric space $(X,d)$ to  equal the $d$-Scott topology. To this end, we need a condition, called the condition (S), for metric spaces.

\begin{defn}[The condition (S), \cite{GN}] A metric space $(X,d)$ is said to satisfy the condition (S) if it satisfies:

(S) For every  directed  set $\{(x_i,r_i)\}_i$ of $(\mathrm{B}X,\sqsubseteq)$ and for every $s\geq 0$, $\{(x_i,r_i)\}_i$ has a join  in $(\mathrm{B}X,\sqsubseteq)$ if and only if   $\{(x_i,r_i+s)\}_i$ has a join in $(\mathrm{B}X,\sqsubseteq)$.    \end{defn}

Metric spaces satisfying the condition (S) are introduced in \cite{GN} as \emph{standard quasi-metric spaces}. In this paper, we do not use the terminology \emph{quasi-metric space}, so, we say that such spaces satisfy the condition (S). It is shown in \cite{GN} that a large class of metric spaces  satisfy the condition (S), including symmetric metric spaces,   Yoneda complete metric spaces, and ordered sets (as metric spaces). A nice property of these  spaces  is that the converse of Lemma \ref{7.4.25} is also true.

\begin{lem}
\label{7.4.26} Let $(X,d)$ be a metric space that satisfies the condition {\rm (S)}. If $(x,r)$ is a join of a directed set $\{(x_i,r_i)\}_i$ in $(\mathrm{B}X,\sqsubseteq)$, then $r=\inf_ir_i$ and $x$ is a Yoneda limit of the underlying net $\{x_i\}_i$. \end{lem}
\begin{proof}The proof is contained in \cite[Lemma 7.4.26]{Goubault}, because the proof therein only requires that if a directed set $\{(x_i,r_i)\}_i$ has a join  in $(\mathrm{B}X,\sqsubseteq)$ then so does $\{(x_i,r_i+s)\}_i$ for every $s\geq -\inf_ir_i$. \end{proof}

\begin{prop}For each metric space  that satisfies the condition {\rm (S)}, the $c$-Scott topology is equal to the $d$-Scott topology. \end{prop}

\begin{proof}We only need to check that for a metric space $(X,d)$   satisfying the condition (S), the $d$-Scott topology is finer than the $c$-Scott topology.

 For a Scott weight $\phi$ of $(X,d)$, define a subset $\mathrm{B}\phi$   of $(\mathrm{B}X,\sqsubseteq)$   by \[\mathrm{B}\phi=\{(x,r)\in \mathrm{B}X\mid \phi(x)\leq r\}.\]
It is clear that $\phi^{-1}(0)=\eta_X(X)\cap\mathrm{B}\phi$. So, if we can show that   $\mathrm{B}\phi$ is Scott closed in $(\mathrm{B}X,\sqsubseteq)$, then the conclusion   follows.

Assume that $\{(x_i,r_i)\}_i$ is a directed set in $\mathrm{B}\phi$ and $(x,r)$ is a join of $\{(x_i,r_i)\}_i$ in $(\mathrm{B}X,\sqsubseteq)$. By Lemma \ref{7.4.26}, $r=\inf_ir_i$ and $x$ is a Yoneda limit of the forward Cauchy net $\{x_i\}_i$. Then  \[\phi(x)\leq\inf_i\sup_{j\geq i}\phi(x_j)\leq\inf_i\sup_{j\geq i}r_j=r,\] showing that $(x,r)\in\mathrm{B}\phi$, hence $\mathrm{B}\phi$ is Scott closed.
\end{proof}

\begin{rem} We don't know whether the $c$-Scott topology coincides with the generalized Scott topology for every metric space. If the answer is yes, then the generalized Scott topology is equal to the $d$-Scott topology for all metric spaces that satisfy the condition (S). This should be compared with  \cite[Exercise 7.4.69]{Goubault} which says that for a Yoneda complete algebraic metric space, the generalized Scott topology is equal to the $d$-Scott topology (also see Corollary \ref{c=g} below). If the answer is no, then the $c$-Scott topology is a new  and a natural topology for metric spaces.   \end{rem}

\section{Scott distance on algebraic metric spaces}
In this section we show that for an algebraic metric space $(X,d)$, the Scott distance of $(X,d)$ is determined  by its compact elements, and it coincides with the approach distance introduced in  Windels \cite{Windels}.

\begin{defn}\label{finite-definition} (\cite{BvBR1998,Goubault}) An element $a$ in a metric space $(X,d)$ is compact if for each forward Cauchy net $\{x_i\}_i$ with a Yoneda limit $x$, $d(a,x)=\inf_i\sup_{j\geq i}d(a,x_j)$.  A metric space $(X,d)$ is algebraic if  every element in $(X,d)$ is a Yoneda limit of a forward Cauchy net consisting of compact elements. \end{defn}

\begin{exmp}(\cite{Goubault}) Every element in $([0,\infty],d_L)$ is compact, hence $([0,\infty],d_L)$ is algebraic.  Every element except $\infty$ is  compact in $([0,\infty],d_R)$ and $\infty$ is the Yoneda limit of the forward Cauchy sequence $\{n\}$, so,  $([0,\infty],d_R)$ is algebraic. \end{exmp}

\begin{prop}\label{finite} An element $a$ in a metric space $(X,d)$ is compact if and only if for each flat weight $\phi$ with a colimit it holds that $d(a,\colim\phi)=\phi(a).$\end{prop}

\begin{proof}
Let $\phi$ be a  flat  weight of $(X,d)$ with a colimit. It follows from Proposition \ref{flat weight} and Proposition \ref{Yoneda limit=colimit} that there exists a forward Cauchy net $\{x_i\}_i$ in $(X,d)$ such that $\phi=\inf_i\sup_{j\geq i}d(-,x_j)$ and that $\colim\phi$ is a Yoneda limit of $\{x_i\}_i$. Then $$ d(a,\colim\phi)= \inf_i\sup_{j\geq i}d(a,x_j) = \phi(a),$$ proving the necessity. Conversely, suppose $\{x_i\}_i$  is a forward Cauchy net with a Yoneda limit $x$. Then $\phi=\inf_i\sup_{j\geq i}d(-,x_j)$ is a flat weight of $(X,d)$ having  $x$ as a colimit. Hence $$d(a,x) = \phi(a)= \inf_i\sup_{j\geq i}d(a,x_j),$$   proving the sufficiency.
\end{proof}

\begin{exmp}
For a metric space $(X,d)$, consider the subspace  $(\mathcal{F}X,\overline{d})$ of $(\CP X,\overline{d})$, where $\mathcal{F}X=\{\phi\in \CP X\mid \phi  {\rm~is~flat}\}$. This separated metric space is  a metric version of the partially ordered set of ideals in an ordered set.   $(\mathcal{F}X,\overline{d})$ is both Yoneda complete and algebraic.\footnote{In the language of enriched category theory \cite{KS05}, $(\mathcal{F}X,\overline{d})$ is the free cocompletion of $(X,d)$ with respect to the class of flat weights. Moreover, every Yoneda complete, algebraic, and separated metric space is of this form.}

For each forward Cauchy net $\{\phi_i\}_i$ in $(\mathcal{F}X,\overline{d})$, the weight   $\phi=\inf_i\sup_{j\geq i}\phi_i$ is flat by \cite[Theorem 7.15]{SV2005},  hence it is a Yoneda limit of $\{\phi_i\}_i$ in $(\mathcal{F}X,\overline{d})$ by  Example \ref{Yoneda limits in PX}. Therefore, $(\mathcal{F}X,\overline{d})$ is Yoneda complete.

Given a flat weight  $\phi$ of $(X,d)$, by Proposition \ref{flat weight}, there is a forward Cauchy net $\{x_i\}_i$ in $(X,d)$ such that $\phi=\inf_i\sup_{j\geq i}d(-,x_i)$. Then $\phi$ is a Yoneda limit of the forward Cauchy net $\{\y(x_i)\}_i$ in $(\mathcal{F}X,\overline{d})$. So, in order to see that $(\mathcal{F}X,\overline{d})$ is   algebraic, it suffices to verify that  for all $x\in X$, $\y(x)$ is compact in $(\mathcal{F}X,\overline{d})$. Let $\{\phi_i\}_i$ be a forward Cauchy net  in $(\mathcal{F}X,\overline{d})$. Since the Yoneda limit of  $\{\phi_i\}_i$ in $(\mathcal{F}X,\overline{d})$ is  given by $\phi=\inf_i\sup_{j\geq i}\phi_j$, then
\[\overline{d}(\y(x),\phi)=\phi(x)=\inf_i\sup_{j\geq i}\phi_j(x)= \inf_i\sup_{j\geq i}\overline{d}(\y(x), \phi_j),\]
hence $\y(x)$ is compact.
\end{exmp}

It is easily seen that an element  $x$ in an ordered set $(P,\leq)$ is   compact if and only if the upper set $P\setminus\ua\!x$ is Scott closed. The following conclusion is a metric version of this fact.
\begin{lem}\label{compact element by regular} An element $b$ in a metric space $(X,d)$ is compact if and only if  $r\ominus d(b,-)$ is a Scott weight for all $r\in[0,\infty]$. \end{lem}

\begin{proof} Suppose $b$ is a compact element of $(X,d)$. It is easy to verify that  $\phi=r\ominus d(b,-)$ is a weight of $(X,d)$, so, it remains to check that for any flat weight $\psi$, $\overline{d}(\psi,\phi)\geq \phi(\colim\psi)$  whenever $\colim\psi$ exists. Since $b$ is compact,  we have $\psi(b)=d(b,\colim\psi)$ by Proposition \ref{finite}. Then
\begin{align*}
  \overline{d}(\psi,\phi)&=\sup_{x\in X}(\phi(x)\ominus \psi(x)) \\
   &=\sup_{x\in X} \big((r\ominus d(b,x))\ominus \psi(x)\big)\\
   &=\sup_{x\in X}\big(r \ominus (d(b,x)+\psi(x))\big)\\
   &= r \ominus \psi(b) \\
   &=  r \ominus d(b,\colim\psi) \\
   &=\phi(\colim\psi).
\end{align*}

Conversely, let $b$ be an element such that $r\ominus d(b,-)$ is a Scott weight for all $r\in[0,\infty]$.  We show that $b$ is compact. By Proposition \ref{finite}, it suffices to check that for every flat weight $\psi$, $\psi(b)=d(b,\colim\psi)$ whenever $\colim\psi$ exists. Since $\colim\y(b)=b$, we have   \[\psi(b)= \overline{d}(\y(b),\psi)\geq  d(b,\colim\psi).\] Let $r=\psi(b)$. Since $r\ominus d(b,-)$ is a Scott weight and \[\overline{d}(\psi,r\ominus d(b,-))=\sup_{x\in X}\big((r\ominus d(b,x))\ominus\psi(x)\big)=r\ominus\psi(b)=0,\] it follows that \[0\geq r\ominus d(b,\colim\psi),\] hence $\psi(b)=r\leq d(b,\colim\psi)$. \end{proof}

\begin{thm}\label{algebraic distance} Let $(X,d)$ be an algebraic metric space and $B$ be the set of compact elements of $(X,d)$. Then \begin{equation}\label{windels}\si(x,A)=\sup_{b\in B}\Big(\inf_{a\in A}d(b,a)\ominus d(b,x)\Big) \end{equation} for every nonempty subset $A\subseteq X$ and   $x\in X$. Hence the Scott distance coincides with the approach distance given in  Windels \cite{Windels}.  \end{thm}

\begin{proof}Write $\si$ for the Scott  distance of $(X,d)$. For each $b\in B$ and $r\in[0,\infty]$, $ r\ominus d(b,-)$ is a regular function of $(X,\sigma)$ by Lemma \ref{compact element by regular}, so, \[\phi=\sup_{b\in B}\Big(\inf_{a\in A}d(b,a)\ominus d(b,x)\Big)\] is a regular function of $(X,\sigma)$.  Since $\phi(a)=0$ for all $a\in A$, it follows that \[\si(x,A)\geq\sup_{b\in B}\Big(\inf_{a\in A}d(b,a)\ominus d(b,x)\Big).\] To see the converse inequality,
we  only need to show  that for each Scott weight $\phi$ of $(X,d)$, if  $\phi(a)=0$ for all $a\in A$, then for all $x\in X$, \[\phi(x)\leq\sup_{b\in B}\Big(\inf_{a\in A}d(b,a)\ominus d(b,x)\Big). \]

Since $(X,d)$ is algebraic, there is  a forward Cauchy net $\{x_i\}_i$ in $B$ with $x$ as a Yoneda limit. Then, by Proposition \ref{Yoneda limit=colimit}, $x$ is a colimit of the flat weight $\inf_i\sup_{j\geq i}d(-,x_j)$. Since $\phi$ is a Scott weight,  one has \begin{align*}\phi(x)&\leq \overline{d}(\inf_i\sup_{j\geq i}d(-,x_j),\phi) \leq \sup_i\inf_{j\geq i} \overline{d} (d(-, x_j),\phi) =\sup_i\inf_{j\geq i}\phi(x_j).\end{align*} Thus, it suffices to show that  \[  \sup_{b\in B}\Big(\inf_{a\in A}d(b,a)\ominus d(b,x)\Big) \geq \sup_i\inf_{j\geq i}  \phi(x_j).\]

For each $\varepsilon>0$ and  $K\in[0,\infty)$ with $K\leq \sup_i\inf_{j\geq i}\phi(x_j)$, since $x$ is a Yoneda limit of $\{x_i\}_i$, there exists an index $k$ such that $d(x_k,x)\leq\varepsilon$  and  $\phi(x_k)\geq K-\varepsilon$.
Since $\phi$ is a weight,   for all $a\in A$,  \[\phi(x_k)\leq d(x_k,a)+\phi(a)=d(x_k,a), \]   hence  \[\inf_{a\in A}d(x_k,a)\ominus d(x_k,x)\geq \phi(x_k)-d(x_k,x)\geq K-2\varepsilon.\]  By arbitrariness of $K$ and $\varepsilon$, we obtain that \[  \sup_{b\in B}\Big(\inf_{a\in A}d(b,a)\ominus d(b,x)\Big) \geq \sup_i\inf_{j\geq i}  \phi(x_j).\]
This completes the proof. \end{proof}

\begin{cor}\label{basis} For an algebraic metric space $(X,d)$, \[\{r\ominus d(b,-)\mid b\in B, r\in[0,\infty]\}\] is a subbasis for the regular functions of $\Sigma(X,d)$. \end{cor}
\begin{proof}Write $\si$ for the Scott  distance of $(X,d)$. Since $\{\si(-,A)\mid A\subseteq X\}$ is a subbasis for the regular functions of $\Sigma(X,d)$, it suffices to check that for every nonempty subset $A\subseteq X$, $\si(-,A)$ belongs to the set of regular functions generated as a subbasis by \[\{r\ominus d(b,-)\mid b\in B, r\in[0,\infty]\}.\]   This follows immediately from Equation (\ref{windels}).
\end{proof}

\begin{exmp}\label{P=Sigma} $\Sigma([0,\infty],d_R)=\mathbb{P}$. Let $\si$ be the Scott distance of $([0,\infty],d_R)$. Since every element except $\infty$ is compact in $([0,\infty],d_R)$,  it holds by Equation (\ref{windels})  that for all $x\in[0,\infty]$ and all nonempty subset $A\subseteq[0,\infty]$, \[\si(x,A)=\sup_{b\in[0,\infty)}\Big(\inf_{a\in A}d_R(b,a)\ominus d_R(b,x)\Big).\]
It is routine to check,   distinguishing whether $\sup A=\infty$, that $\Sigma([0,\infty],d_R)=\mathbb{P}$.
\end{exmp}

\begin{cor}\label{c=g} {\rm (\cite[Proposition 3.4]{Windels})} For an algebraic metric space $(X,d)$, the $c$-Scott topology is equal to the generalized Scott topology. \end{cor}
\begin{proof}Since for each compact element $b$ in $(X,d)$ and $r\in[0,\infty]$, \[(r\ominus d(b,-))^{-1}(0,\infty]= \{x\mid d(b,x)<r\}= B(b, r),\] then by Corollary \ref{basis}, the set \[\{B(b,r)\mid b ~\text{is compact}, r>0\}\] of open balls is a subbasis for the  $c$-Scott topology.  It is shown in \cite[Proposition 6.5]{BvBR1998} that this set of open balls is also a basis for the generalized Scott topology, hence these two topologies are equal.\end{proof}

Since every Yoneda complete metric space satisfies the condition (S), it follows that for a Yoneda complete algebraic metric space, the $d$-Scott topology, which is equal to the $c$-Scott topology, coincides with the  generalized Scott topology, as asserted in \cite[Exercise 7.4.69]{Goubault}.

Let $\{(X_i,d_i)\}_{i\in I}$ be a family  of metric spaces  and $\prod_i (X_i, d_i)$  their product. Suppose that $\{(x_i)^\lam\}_\lam$ is a forward Cauchy net in  $\prod_i (X_i, d_i)$. By  \cite[Lemma 7.4.13 and Lemma 7.4.15]{Goubault},  $(x_i)_i$ is a Yoneda limit of  $\{(x_i)^\lam\}_\lam$ in $\prod_i (X_i, d_i)$ if and only if  for all $i\in I$, $x_i$ is a Yoneda limit of the forward Cauchy net $\{{x_i}^\lam\}_\lam$ in $(X_i,d_i)$. In particular, if  $(X_i,d_i)$    is Yoneda complete for all $i\in I$ then so is $\prod_i (X_i, d_i)$.

A metric space $(X,d)$ is said to have a bottom element if there is an element $\bot$ in $X$ such that $d(\bot,x)=0$ for all $x\in X$.  The following conclusion is   \cite[Exercise 7.4.71]{Goubault}.

\begin{prop}\label{product of algebraic} For a family $\{(X_i,d_i)\}_i$ of algebraic metric spaces with   bottom elements, the product space $\prod_i (X_i, d_i)$ is algebraic. An element $b=(b_i)_i$ is compact in $\prod_i (X_i, d_i)$ if and only if every $b_i$ is compact  and for each $\varepsilon>0$, there is a finite subset $J_\varepsilon$ of $I$ such that $d_i(b_i,\bot_i)\leq \varepsilon$ whenever  $i\notin J_\varepsilon$. \end{prop}

The following conclusion will be needed in the proof of the main result in next section, Theorem \ref{main}.
\begin{prop}\label{Scott distance preserves product} Let $(X,d)$ be an algebraic metric space with a bottom element $\bot$. Then  for each non-empty set $I$, $\Sigma((X,d)^I)= (\Sigma(X,d))^I$. In particular,  $\Sigma(([0,\infty],d_R)^I)= \mathbb{P}^I$. \end{prop}
\begin{proof}Write $p_i:X^I\lra X$ for the $i$th projection and $\rho$ for the   metric of the product $(X,d)^I$, i.e., $\rho((a_i),(b_i))=\sup_{i\in I}d(a_i,b_i)$. Let $\si$ denote the Scott  distance of $(X^I,\rho)$ and $\delta$ the distance of the product space $(\Sigma(X,d))^I$. We must show that $\sigma=\delta$. To this end, we show that they have the same regular functions.

Let $B$ be the set of compact elements in $(X,d)$. By  Corollary \ref{basis}, \[\{r\ominus d(a,-)\mid r\in[0,\infty], a\in B\}\] is a subbasis for the regular functions of $\Sigma(X,d)$. Hence \[\{(r\ominus d(a,-))\circ p_i\mid r\in[0,\infty], a\in B, i\in I\}\] is a subbasis for the regular functions of $(X^I,\delta)$.

For each $a\in B$ and each $i$, define $b^i=(b^i_j)\in X^I$ by $b^i_i=a$ and $b^i_j=\bot$ whenever $j\not=i$. Then $b^i$ is a compact element in $(X^I,\rho)$. Since \[r\ominus \rho(b^i,x) =r\ominus d(a, x_i)= (r\ominus d(a,-))\circ p_i(x),\] it follows that every regular function of $(X^I,\delta)$ is a regular function of $(X^I,\sigma)$.

To see that every regular function of $(X^I,\sigma)$ is a regular function of  $(X^I,\delta)$, it suffices to show that for each compact element $b=(b_i)$ of $(X^I,\rho)$ and each $r\in[0,\infty]$, \[r\ominus \rho(b,-)\] is a regular function of  $(X^I,\delta)$.

Let $\varepsilon>0$. Without loss of generality, we assume that $\varepsilon\leq r$. Since $b$ is compact in $(X^I,\rho)$, by Proposition \ref{product of algebraic}, $b_i\in B$ for all $i$ and there is a finite subset $J$ of $I$ such that $d(b_i,\bot)\leq\varepsilon$ whenever $i\notin J$. For each $i\in J$,  \[(r\ominus d(b_i,-))\circ p_i \] is a regular function of $(X^I,\delta)$, hence \[f=\min\{ (r\ominus d(b_i,-))\circ p_i\mid i\in J\}\] is a regular function of $(X^I,\delta)$. For each $x\in X^I$, since \[f(x)=\min \{(r\ominus d(b_i,-))\circ p_i(x)\mid i\in J\}=r\ominus\max\{d(b_i, x_i)\mid i\in J\}, \] then \[r\ominus \rho(b,x)=r\ominus\sup_{i\in I}d(b_i, x_i)\leq f(x).\] Let \[f_\varepsilon=\min\{ ((r-\varepsilon)\ominus d(b_i,-))\circ p_i\mid i\in J\}.\] Then $f_\varepsilon$ is a regular function of $(X^I,\delta)$ such that $f_\varepsilon\leq f\leq f_\varepsilon+\varepsilon$. Since $d(b_i, x_i)\leq \varepsilon$ whenever $i\notin J$, it follows that
\[f_\varepsilon(x)\leq r\ominus \rho(b,x)\leq f(x).\] By arbitrariness of $\varepsilon$, we obtain that $r\ominus \rho(b,-)=\sup_{\varepsilon>0}f_\varepsilon$, hence $r\ominus \rho(b,-)$ is a regular function of $(X^I,\delta)$.
\end{proof}

The argument of the above proposition can be applied to show that if $\{(X_i,d_i)\}_i$ is a family of algebraic metric spaces with bottom elements, then $\prod_i\Sigma(X_i,d_i)= \Sigma\prod_i(X_i,d_i)$.  Consequently, the  $c$-Scott topology (= the generalized Scott topology) on the product metric space $\prod_i(X_i,d_i)$ is equal to the product topology of the $c$-Scott topologies on the factor spaces.

\section{Scott distance on   continuous metric spaces}
In 1972, Scott \cite{Scott72} proved that the specialization order functor establishes an isomorphism between the category of injective $T_0$ topological spaces and that of continuous lattices.  In this section, we investigate whether we have a metric version of this isomorphism. It should be pointed out that there is a quite different approach to this topic, see \cite{GH,Hof2011,Hof2013} for details.

An approach space $(X,\delta)$ is $T_0$ if $x=y$ whenever $\delta(x,\{y\})=\delta(y,\{x\})=0$.
A contraction $e: (X,\delta_X)\lra (Y,\delta_Y)$   is an embedding if $\delta_X(x,A)=\delta_Y(e(x),e(A))$ for all $A\subseteq X$ and $x\in X$.
A $T_0$ approach space $(Z,\delta_Z)$ is   injective  if for every embedding  $e: (X,\delta_X) \lra (Y,\delta_Y)$ and every contraction $f: (X,\delta_X) \lra(Z,\delta_Z)$, there exists a contraction $f^*: (Y,\delta_Y) \lra(Z,\delta_Z)$ that extends $f$, i.e., $f=f^*\circ e$.

The following lemma follows immediately from  \cite[Theorem 1.10.7]{RL97} which implies that for an embedding $(X,\delta_X) \lra (Y,\delta_Y)$ in the category {\sf App}, the regular functions on $X$ are precisely the restrictions of the regular functions on $Y$. The conclusion has also been proved in a more general context in \cite{HT2010}.

\begin{lem}  The approach space $\mathbb{P}=\Sigma([0,\infty],d_R)$ is injective. \end{lem}

For each $T_0$ approach space $(X,\delta)$, the map \[e:(X,\delta)\lra\mathbb{P}^{2^X},\quad e(x)=(\delta(x,A))_{A\in 2^X} \]
is an embedding. Hence, every $T_0$ approach space can be embedded in some power of $\mathbb{P}$.

\begin{prop}\label{injective = retract} A $T_0$ approach space is injective if and only if it is a retract of some power of $\mathbb{P}$.\end{prop}

Now we turn to the metric analogy of continuous lattices:  cocomplete and continuous separated metric spaces.

Let $(X,d)$ be a Yoneda complete metric space. By Proposition \ref{flat weight} and Proposition \ref{Yoneda limit=colimit} we know that each flat weight of $(X,d)$ has a colimit. The correspondence $\phi\mapsto\colim\phi$ defines a map $\colim:(\mathcal{F}X,\overline{d})\lra(X,d)$. This map is in fact a left adjoint of the Yoneda embedding $\y:(X,d)\lra(\mathcal{F}X,\overline{d})$   \cite{KW2011}.

\begin{defn}\label{def of continuous metric} (\cite{KW2011})
A metric space $(X,d)$ is continuous if it is Yoneda complete and the left adjoint $\colim:(\mathcal{F}X,\overline{d})\lra (X,d)$ of the Yoneda embedding $\mathbf{y}:(X,d) \lra (\mathcal{F}X,\overline{d})$ has a left adjoint,  denoted by $\thda:(X,d) \lra (\mathcal{F}X,\overline{d})$.\end{defn}

The following conclusion provides an important class of continuous metric spaces.

\begin{prop}\label{algebraic} Yoneda complete algebraic metric spaces are continuous. \end{prop}
\begin{proof}
We show that the left adjoint $\colim: (\mathcal{F}X,\overline{d}) \lra (X,d)$ of the Yoneda embedding $\y: (X,d) \lra (\mathcal{F}X,\overline{d})$  has a left adjoint. For each  $a\in X$, take a forward Cauchy net $\{a_i\}_i$ of compact elements in $(X,d)$ with $a$ as a Yoneda limit and let $\thda\!a$ be the Yoneda limit of the forward Cauchy net $\{\y(a_i)\}_i$ in $(\mathcal{F}X,\overline{d})$. Then for every flat weight $\phi$ of $(X,d)$,
\begin{align*}
\overline{d}(\thda\!a,\phi)
&=\inf_i\sup_{j\geq i}\overline{d}(\y(a_j), \phi)&(\thda\! a{\rm~is~a~Yoneda~limit~of~} \{\y(a_i)\}_i)\\
&=\inf_i\sup_{j\geq i}\phi(a_j)&({\rm Yoneda~lemma})
\\
&=\inf_i\sup_{j\geq i}d(a_j, \colim\phi)&( a_j{\rm~is~compact}) \\
&=d(a,\colim\phi),&(a{\rm~is~a~Yoneda~limit~of~} \{a_i\}_i)
\end{align*}
hence $\thda$ is a left adjoint of $\colim$.
\end{proof}

The following lemma is proved in \cite{GN}  using a characterization of continuous metric spaces in terms of formal balls. For sake of self-containment, we include a direct proof here.
\begin{lem}\label{retracts are continuous} {\rm(\cite[Proposition 7.1]{GN})} In the category of metric spaces and Yoneda continuous maps, retracts of  continuous metric spaces  are continuous. \end{lem}

\begin{proof}Suppose that $(X,d_X)$ is a continuous metric space, $r:(X,d_X)\lra (Y,d_Y)$ and  $s:(Y,d_Y)\lra(X,d_X)$ are Yoneda continuous maps, and that $r\circ s=1$.  We must  show that $(Y,d_Y)$ is continuous.

First of all, we list here some facts about $r$ and $s$:  (i) $s$ is isometric; (ii)  $\overline{r}\circ\overline{s} =1$; and (iii)     $\overline{r}(\phi\circ r)=\phi$ for all $\phi\in\CP Y$ (verification is left to the reader).

Write $\colim_X:(\mathcal{F}X,\overline{d_X})\lra (X,d_X)$ for the left adjoint of  the Yoneda embedding $\y_X:(X,d_X)\lra(\mathcal{F}X,\overline{d_X})$, and  $\thda_X:(X,d_X)\lra(\mathcal{F}X,\overline{d_X})$  for the left adjoint of $\colim_X$. We prove the conclusion in two steps.

\textbf{Step 1}. $(Y,d_Y)$ is Yoneda complete. We leave it to the reader to check that for each flat weight $\phi$ of $(Y,d_Y)$,   $r\circ \colim_X\circ \overline{s}(\phi)$ is a colimit of $\phi$.

\textbf{Step 2}. $(Y,d_Y)$ is continuous. We show that $\thda_Y=\overline{r}\circ \thda_X\circ s$ is left adjoint to   $\colim_Y:(\CF Y, \overline{d_Y})\lra (Y,d_Y)$. That is, $\overline{d_Y}(\thda_Y\!(y),\phi)= d_Y(y,\colim_Y(\phi))$ for all $y\in Y$ and $\phi\in\CF Y$. On one hand,
\begin{align*}
  \overline{d_Y}(\thda_Y\!(y),\phi)
  &= \overline{d_Y}(\overline{r}\circ \thda_X\circ s(y),\overline{r}\circ\overline{s}(\phi)) &(\overline{r}\circ\overline{s}=1) \\
  &\leq \overline{d_X}(\thda_X\circ s(y),\overline{s}(\phi))& (\overline{r}~\text{is~non-expansive})\\
  &=d_X(s(y),\colim_X\circ\overline{s}(\phi))&(\thda_X {\rm~is~ left~adjoint~to~} \colim_X)\\
  &=d_X(s(y),s\circ \colim_Y(\phi))&(s{\rm~is~Yoneda~continuous})\\
  &=d_Y(y,\colim_Y(\phi)).&(s {\rm~is~isometric})
  \end{align*}
On the other hand,
\begin{align*}
  \overline{d_Y}(\thda_Y\!(y),\phi)& =\overline{d_Y}(\overline{r}\circ \thda_X\circ s(y),\phi) \\
  &=\overline{d_X}(\thda_X\circ s(y),\phi\circ r)&({\rm Lemma~}\ref{left kan})\\
  &=d_X(s(y),\colim_X(\phi\circ r))&(\thda_X {\rm~is~left~adjoint~to~} \colim_X)\\
  &\geq d_Y(y,r\circ \colim_X(\phi\circ r))&(r~\text{is non-expansive and}~r\circ s=1)\\
  &=d_Y(y,\colim_Y(\overline{r}(\phi\circ r)))&(r{\rm~is~Yoneda~continuous})\\
  &=d_Y(y,\colim_Y(\phi)). &(\overline{r}(\phi\circ r)=\phi)
\end{align*}

This completes the proof.
\end{proof}

Suppose $(X,d)$ is a separated continuous metric space. Both $\colim:(\mathcal{F}X,\overline{d})\lra (X,d)$ and $\thda:(X,d) \lra (\mathcal{F}X,\overline{d})$ are left adjoints,  hence both of them are Yoneda continuous. Since the composite $\colim\circ\thda$ is the identity map,   $(X,d)$ is a retract of the Yoneda complete algebraic metric space $(\mathcal{F}X,\overline{d})$ in the category of metric spaces and Yoneda continuous maps. Because every functor  preserves retracts and the $d$-Scott, $c$-Scott and   generalized Scott topologies coincide for a Yoneda complete algebraic metric space, therefore Corollary \ref{c=g} holds for all separated continuous metric spaces, that is to say, the $d$-Scott,   $c$-Scott and  generalized Scott topologies on such a space coincide with each other.

\begin{defn} (\cite{BvBR1998}) A metric space $(X,d)$ is  cocomplete  if each weight of $(X,d)$ has a colimit.\end{defn}

It is trivial that a metric space $(X,d)$ is  cocomplete if and only if the Yoneda embedding $\y:(X,d)\lra(\CP X,\overline{d})$ has a left adjoint \cite{BvBR1998}.

The following  examples of cocomplete metric spaces are sort of folklore in category theory.
\begin{exmp}\label{cocomplete examples} \begin{enumerate}[label={\rm(\arabic*)}] \setlength{\itemsep}{-2pt} \item  $([0,\infty],d_L)$ is cocomplete. For each weight $\phi$ of $([0,\infty],d_L)$, \[\colim\phi=\inf\limits_{x\in[0,\infty]}(\phi(x)+x).\]

\item $([0,\infty],d_R)$ is cocomplete. For each weight $\psi$ of $([0,\infty],d_R)$, \[\colim\psi=\sup\limits_{x\in[0,\infty]}(x\ominus\psi(x)).\]
\item  For each set $I$, both $([0,\infty],d_L)^I$ and $([0,\infty],d_R)^I$ are cocomplete.

\item  For every metric space $(X,d)$, $(\CP X,\overline{d})$ is cocomplete. For each weight $\Phi$ of $(\CP X,\overline{d})$, \[\colim\Phi=\inf\limits_{\phi\in\CP X}(\Phi(\phi)+\phi).\] \end{enumerate}    \end{exmp}

\begin{prop}\label{retracts are complete} Every retract  of a cocomplete metric space   in the category {\sf Met} is cocomplete. \end{prop}

The following conclusion is a metric analogy of the fact that every continuous lattice is a retract of some powerset in the category of ordered sets and Scott continuous maps.

\begin{prop}\label{3.30}Cocomplete and continuous   separated metric spaces are exactly retracts of  powers of $([0,\infty],d_L)$ in the category   of   metric spaces and Yoneda continuous maps. \end{prop}

\begin{proof}Sufficiency is contained in    Proposition \ref{retracts are complete} and Lemma \ref{retracts are continuous}, it remains to prove the necessity. Given a cocomplete, continuous and separated metric space $(X,d)$, let $s$ be the  composite
\[(X,d)\stackrel{\thda}{\lra} (\CF X, \overline{d})\lra(\CP X, \overline{d}) \lra([0,\infty],d_L)^X,\] where the latter two arrows are inclusions. Since both $\thda:(X,d)\lra(\CF X, \overline{d})$ and the inclusion $(\CP X, \overline{d})\lra([0,\infty],d_L)^X$ are left adjoints,   and since the inclusion $(\CF X, \overline{d})\lra(\CP X, \overline{d})$ is Yoneda continuous, it follows that $s:(X,d)\lra([0,\infty],d_L)^X$ is Yoneda continuous.

Let $p$ be the discrete metric on $X$, i.e., $p(x,x)=0$ and $p(x,y)=\infty$ whenever $x\not=y$. Then $([0,\infty],d_L)^X$ is exactly the metric space $(\CP X, \overline{p})$ of weights of $(X,p)$. Since the identity $1_X:(X,p)\lra(X,d)$ is non-expansive, the map $\overline{1_X}:([0,\infty],d_L)^X\lra(\CP X, \overline{d})$ is a left adjoint by Lemma \ref{left kan}. Let  $r$ be the composite \[\colim\circ \overline{1_X}:([0,\infty],d_L)^X\lra(\CP X, \overline{d})\lra(X,d).\] Then $r$, as a left adjoint, is Yoneda continuous.

Finally, since $r\circ s$ is the identity map on $(X,d)$, it follows that $(X,d)$ is a retract of $([0,\infty],d_L)^X$ in the category   of   metric spaces and Yoneda continuous maps. \end{proof}

It is well known that a continuous lattice together with its  Scott topology is an injective space \cite{Scott72,domains}, but, this is not true in the metric setting.

\begin{exmp}\label{5.8}  The approach space $\Sigma([0,\infty],d_L)$ is not injective. Since $([0,\infty],d_L)$ is Smyth complete, the Scott distance and the Alexandroff distance on $([0,\infty],d_L)$ coincide by Corollary \ref{Smyth}. Using this fact, it can be checked that $\Sigma([0,\infty],d_L)$ coincides with the approach space $[0,\infty]^{\rm op}$ in \cite[Example 4.14]{GH},  hence it is not injective. We also include here a direct verification for convenience of the reader.
Suppose on the contrary that $\Sigma([0,\infty],d_L)$ is injective. Consider the subspace $\{0,\infty\}$ of $\mathbb{P}$. Define \[f:\{0,\infty\}\lra\Sigma([0,\infty],d_L)\] by $f(0)=\infty$ and $f(\infty)=0$. Then $f$ is a contraction, so, there is a contraction $f^*:\mathbb{P}\lra\Sigma([0,\infty],d_L)$ that extends $f$. Since $\mathbb{P}=\Sigma([0,\infty],d_R)$, it follows   that $f^*:([0,\infty],d_R)\lra([0,\infty],d_L)$ is Yoneda continuous, in particular,   non-expansive. Thus, for all $x<\infty$, one has
\[x=d_R(x,0)\geq d_L(f^*(x),f^*(0))=d_L(f^*(x),\infty).  \] Therefore, $f^*(x)=\infty$ for all $x<\infty$.
Since $\infty$ is a Yoneda limit of the forward Cauchy sequence $\{n\}$  in $([0,\infty],d_R)$, it follows that $0=f^*(\infty)$ is a Yoneda limit of the constant sequence with value $\infty$ in $([0,\infty],d_L)$, a contradiction.
\end{exmp}

Therefore,  a cocomplete and continuous separated metric space together with the Scott distance need not be an injective approach space. But, every $T_0$ injective approach space must be of this form.

\begin{thm}\label{main} Let $(X,\delta)$ be a $T_0$ injective approach space. Then  $\Omega(X,\delta)$ is a cocomplete and continuous separated  metric space and $(X,\delta)=\Sigma\Omega(X,\delta)$. \end{thm}

\begin{proof}  By injectivity of $(X,\delta)$ and Proposition \ref{injective = retract},    there exist contractions $s:(X,\delta)\lra \mathbb{P}^I$ and $r: \mathbb{P}^I\lra(X,\delta)$ such that $r\circ s=1$.  We finish the proof in four steps.

\textbf{Step 1}. $\Omega(X,\delta)$ is separated. This is easy since $ (X,\delta) $ is $T_0$.

\textbf{Step 2}. $\Omega(X,\delta)$ is  cocomplete. By Proposition \ref{Scott distance preserves product}, we have \[\Omega(\mathbb{P}^I)=\Omega\Sigma (([0,\infty],d_R)^I)  = ([0,\infty],d_R)^I,\]  then $\Omega(\mathbb{P}^I)$ is cocomplete by Example \ref{cocomplete examples}(3). So, as a retract of $\Omega(\mathbb{P}^I)$, $\Omega(X,\delta)$ is cocomplete.

\textbf{Step 3}. $\Omega(X,\delta)$ is continuous.
 Since  $\Omega(\mathbb{P}^I)= ([0,\infty],d_R)^I $   is continuous,  it suffices,  by Lemma  \ref{retracts are continuous}, to show that both the non-expansive maps $r:([0,\infty],d_R)^I\lra\Omega(X,\delta)$ and $s:\Omega(X,\delta)\lra([0,\infty],d_R)^I$  are Yoneda continuous.

Since \[s\circ r:\Sigma(([0,\infty],d_R)^I) \lra\Sigma(([0,\infty],d_R)^I)\] is a contraction, it follows,  by Theorem \ref{contraction=Yoneda continuous}, that \[s\circ r:([0,\infty],d_R)^I\lra([0,\infty],d_R)^I\] is Yoneda continuous.

Now we show that
$r:([0,\infty],d_R)^I\lra\Omega(X,\delta)$ is Yoneda continuous.
Suppose $a$ is a Yoneda limit of  a forward Cauchy net $\{a_\lam\}_\lam$ in $([0,\infty],d_R)^I$. By Yoneda continuity of  $s\circ r$, we obtain that $\{s\circ r(a_\lam)\}_\lam$ is a forward Cauchy net in $([0,\infty],d_R)^I$ having $s\circ r(a)$ as a Yoneda limit. Since $s:\Omega(X,\delta)\lra([0,\infty],d_R)^I$ is isometric, it follows that $r(a)$ is a Yoneda limit of $\{r(a_\lam)\}_\lam$ in $\Omega(X,\delta)$, hence $r:([0,\infty],d_R)^I\lra\Omega(X,\delta)$ is Yoneda continuous.

Next, we show that $s:\Omega(X,\delta)\lra([0,\infty],d_R)^I$  is Yoneda continuous. That is, if $x$ is a Yoneda limit of a forward Cauchy net $\{x_\lam\}_\lam$   in $\Omega(X,\delta)$,  then $s(x)$ is a Yoneda limit of $\{s(x_\lam)\}_\lam$.  Since  $([0,\infty],d_R)^I$ is Yoneda complete, the  forward Cauchy net $\{s(x_\lam)\}_\lam$   has a Yoneda limit, say $y$. By  Yoneda continuity of $r$, $r(y)$ is a Yoneda limit of $\{r\circ s(x_\lam)\}_\lam=\{x_\lam\}_\lam$ in $\Omega(X,\delta)$, hence $r(y)=x$ by separatedness of $\Omega(X,\delta)$. Then, appealing to the Yoneda continuity of $s\circ r$,  we obtain that $s\circ r(y)$ is a Yoneda limit of $\{s\circ r\circ s(x_\lam)\}_\lam=\{s(x_\lam)\}_\lam$, hence $y=s\circ r(y)=s(x)$, showing that $s(x)$ is a Yoneda limit of $\{s(x_\lam)\}_\lam$.

\textbf{Step 4}.  $(X,\delta)=\Sigma\Omega(X,\delta)$. Since both     $r:([0,\infty],d_R)^I\lra\Omega(X,\delta)$ and $s:\Omega(X,\delta)\lra([0,\infty],d_R)^I$   are Yoneda continuous, then  both $r:\mathbb{P}^I\lra\Sigma\Omega(X,\delta)$   and $ s:\Sigma\Omega(X,\delta)\lra\mathbb{P}^I$ are contractions by Theorem \ref{contraction=Yoneda continuous}. Therefore, both \[1_X=r\circ s:(X,\delta)\lra \mathbb{P}^I\lra\Sigma\Omega(X,\delta)\] and \[1_X=r\circ s:\Sigma\Omega(X,\delta)\lra\mathbb{P}^I\lra(X,\delta)\] are contractions, showing that $(X,\delta)=\Sigma\Omega(X,\delta)$.
\end{proof}

\begin{cor}\label{5.7} The category   of injective $T_0$ approach spaces and contractions is isomorphic to the full subcategory of ${\sf Met}^\ua$ consisting of retracts of   powers of $([0,\infty],d_R)$. \end{cor}

\begin{proof} Write {\sf InjApp} for the category  of injective $T_0$ approach spaces and contractions  and write {\sf C} for the full subcategory of ${\sf Met}^\ua$ consisting of retracts of   powers of $([0,\infty],d_R)$.
The argument of Theorem \ref{main} shows that for each injective  $T_0$ approach space $(X,\delta)$, $\Omega(X,\delta)$ is an object in {\sf C}. This fact together with Lemma \ref{4.3} and Theorem \ref{contraction=Yoneda continuous} show  that $\Sigma:{\sf C}\lra{\sf InjApp}$ and   $\Omega: {\sf InjApp}\lra{\sf C}$ are inverse to each other. \end{proof}

The relationship between injective $T_0$ approach spaces and cocomplete and continuous separated metric spaces is summarized as follows. Cocomplete and continuous separated metric spaces  are  retracts of powers of $([0,\infty],d_L)$ in the category ${\sf Met}^\ua$;  injective $T_0$ approach spaces are essentially retracts of powers of  $([0,\infty],d_R)$ in ${\sf Met}^\ua$. The asymmetry between $([0,\infty],d_R)$ and $([0,\infty],d_L)$ accounts for the failure of the metric version of the isomorphism of Scott between injective $T_0$ spaces and continuous lattices.

\vskip 5pt

\noindent{\bf Acknowledgement} The authors thank sincerely the referee for the  thorough analysis of the paper and  the very helpful  comments and   suggestions.

\end{document}